\documentclass[11pt,amsfonts, epsfig,reqno]{amsart}
\usepackage{amsmath, amscd, amssymb, amsthm}
\usepackage{marginnote}
\usepackage[margin=1.29in]{geometry} 

\usepackage{graphpap, color}
\usepackage{mathrsfs}
\usepackage{pstricks}
\usepackage{color}
\usepackage{cancel}
\usepackage[mathscr]{eucal}

\usepackage{color}
\usepackage{cancel}
\usepackage{verbatim}
\usepackage{latexsym}
\usepackage[all]{xy}
\usepackage{graphicx}   
\usepackage{tikz}
\usepackage{caption}
\usepackage[colorlinks=true,linkcolor=blue,citecolor=blue,urlcolor=blue]{hyperref}

\theoremstyle{definition}

\def\locg{\Theta}

\def\sL{\mathscr L} \def\sN{\mathscr N}

\def\tt{\tau}
\def\pr{{\text{pr}}}

\usepackage{upgreek}
\usepackage{wasysym}
\usepackage{ esint }
\def\gndzo{G_{g,n,d}^{[0,1]}}

\numberwithin{equation}{section}
\usepackage[T1]{fontenc}
\usetikzlibrary{calc,decorations.pathreplacing,shadows.blur}

\providecommand{\MSPNeckLabel}{\mathbb P(1,5)}

\definecolor{MSPink}{HTML}{282638}
\definecolor{MSPindigo}{HTML}{2D2988}
\definecolor{MSPlilac}{HTML}{ECE7FC}
\definecolor{MSProse}{HTML}{FFE8EB}
\definecolor{MSPpurple}{HTML}{8B7FC2}
\definecolor{MSPline}{HTML}{5E5877}
\definecolor{MSPorange}{HTML}{F58A09}
\definecolor{MSPochre}{HTML}{C56508}
\definecolor{MSPblue}{HTML}{CDD9E9}

\tikzset{
  msp outline/.style={draw=MSPline,line width=.65pt},
  msp mesh/.style={draw=MSPpurple!48,line width=.32pt},
  msp guide/.style={draw=MSPline!48,line width=.55pt,dash pattern=on .5pt off 3.3pt,line cap=round},
  msp label/.style={text=MSPink,font=\large,inner sep=2pt},
  msp small/.style={text=MSPink,font=\normalsize,inner sep=2pt},
  msp orange/.style={draw=MSPorange,line width=1.65pt,line cap=round,line join=round}
}

\newcommand{\MSPGeometry}{%
  \draw[msp guide] (-4.62,2.35)--(-2.24,2.35);
  \draw[msp guide] (-4.62,.60)--(-.18,.60);
  \draw[msp guide] (-4.62,-1.65)--(-2.27,-1.65);
  \node[msp label] at (-4.91,2.35) {$\infty$};
  \node[msp label] at (-4.91,.60) {$1$};
  \node[msp label] at (-4.91,-1.65) {$0$};
  \draw[decorate,decoration={brace,amplitude=7pt},draw=MSPline,line width=.85pt]
    (-5.38,.60)--(-5.38,2.35);
  \draw[decorate,decoration={brace,amplitude=7pt},draw=MSPline,line width=.85pt]
    (-5.38,-1.65)--(-5.38,.60);
  \node[msp small,align=center] at (-6.33,1.49) {$[1,\infty]$\\[-1pt]theory};
  \node[msp small,align=center] at (-6.33,-.55) {$[0,1]$\\[-1pt]theory};

  \foreach \xx/\yy/\dx in {-2.02/-1.74/-.36,-1.54/-1.98/-.27,-.93/-2.13/-.17,
                            -.33/-2.20/-.06,.34/-2.20/.06,.94/-2.13/.17,
                            1.55/-1.98/.27,2.02/-1.74/.36}{
    \draw[MSPline!68,line width=.60pt]
      (\xx,\yy) .. controls ({\xx+.35*\dx},{\yy-.18})
      and ({\xx+.78*\dx},{\yy-.40}) .. ({\xx+\dx},{\yy-.57});
  }
  \path[shade,left color=MSPblue!72,right color=MSPlilac!55,shading angle=12]
    (-2.25,-1.37)
      .. controls (-1.73,-1.01) and (1.73,-1.01) .. (2.25,-1.37)
      .. controls (2.08,-1.86) and (1.27,-2.17) .. (0,-2.20)
      .. controls (-1.27,-2.17) and (-2.08,-1.86) .. (-2.25,-1.37);
  \draw[MSPline!58,densely dashed,line width=.48pt]
    (-2.25,-1.37) .. controls (-1.73,-1.01) and (1.73,-1.01) .. (2.25,-1.37);
  \draw[msp outline] (-2.54,-1.69)--(-2.25,-1.37)
    .. controls (-2.08,-1.86) and (-1.27,-2.17) .. (0,-2.20)
    .. controls (1.27,-2.17) and (2.08,-1.86) .. (2.25,-1.37)--(2.55,-1.65);
  \draw[MSPline!30,line width=.4pt] (-1.84,-1.56)
    .. controls (-.85,-1.94) and (.85,-1.94) .. (1.84,-1.56);
  \path[fill=MSPline!11] (0,-1.76) ellipse[x radius=1.14,y radius=.28];

  \draw[MSPorange!65,densely dashed,line width=.65pt]
    (1.02,-1.65) arc[start angle=0,end angle=180,x radius=1.02,y radius=.28];
  \path[shade,left color=MSPochre!42,middle color=orange!9,
        right color=MSPorange!66,shading angle=0]
    (0,.60)--(-1.02,-1.65)
    .. controls (-.74,-2.01) and (.74,-2.01) .. (1.02,-1.65)--cycle;
  \foreach \ang in {205,230,255,280,305,330,350}{
    \draw[MSPochre!27,line width=.35pt]
      (0,.60)--({1.02*cos(\ang)},{-1.65+.28*sin(\ang)});
  }
  \foreach \h in {.28,.50,.72}{
    \draw[MSPorange!28,line width=.35pt]
      ({1.02*\h},{.60-2.25*\h})
      arc[start angle=0,end angle=-180,x radius={1.02*\h},y radius={.28*\h}];
  }
  \draw[msp orange] (-1.02,-1.65)--(0,.60)--(1.02,-1.65);
  \path[fill=white!84!MSPorange] (0,-1.65) ellipse[x radius=1.02,y radius=.28];
  \draw[MSPorange!70,line width=.9pt]
    (1.02,-1.65) arc[start angle=0,end angle=180,x radius=1.02,y radius=.28];
  \draw[MSPochre,line width=2.2pt]
    (-1.02,-1.65) arc[start angle=180,end angle=360,x radius=1.02,y radius=.28];
  \draw[orange!60!yellow,line width=.65pt]
    (-.99,-1.66) arc[start angle=180,end angle=360,x radius=.99,y radius=.25];

  \begin{scope}[shift={(0,2.35)},yscale=.24]
  \begin{scope}[shift={(0,-2.35)}]
  \path[shade,left color=MSPpurple!68,middle color=white,
        right color=MSPindigo!38,shading angle=0]
    (-2.08,2.35)
      .. controls (-1.06,2.23) and (-.22,1.78) .. (-.095,1.29)
      .. controls (-.055,1.21) and (.055,1.21) .. (.095,1.29)
      .. controls (.22,1.78) and (1.06,2.23) .. (2.08,2.35)
      arc[start angle=0,end angle=180,x radius=2.08,y radius=.43];
  \path[shade, top color=MSPlilac!42,bottom color=MSPpurple!35]
    (0,2.35) ellipse[x radius=2.08,y radius=.43];
  \draw[MSPline!58,densely dashed,line width=.55pt]
    (2.08,2.35) arc[start angle=0,end angle=180,x radius=2.08,y radius=.43];
  \foreach \rr/\yy in {.45/1.72,.92/2.01,1.48/2.22}{
    \draw[msp mesh] (-\rr,\yy)
      arc[start angle=180,end angle=360,x radius=\rr,y radius={.205*\rr}];
  }
  \foreach \ang in {202,226,250,274,298,322,344}{
    \draw[msp mesh]
      (.07,1.26) .. controls ({.32*cos(\ang)},{1.78+.08*sin(\ang)})
      and ({1.20*cos(\ang)},{2.14+.25*sin(\ang)})
      .. ({2.08*cos(\ang)},{2.35+.43*sin(\ang)});
  }
  \draw[msp outline] (-2.08,2.35)
    .. controls (-1.06,2.23) and (-.22,1.78) .. (-.095,1.29);
  \draw[msp outline] (2.08,2.35)
    .. controls (1.06,2.23) and (.22,1.78) .. (.095,1.29);
  \draw[msp outline]
    (-2.08,2.35) arc[start angle=180,end angle=360,x radius=2.08,y radius=.43];
  \draw[white,opacity=.66,line width=.9pt]
    (-1.92,2.38) arc[start angle=180,end angle=255,x radius=1.92,y radius=.39];

  \end{scope}
  \end{scope}

  \draw[MSPochre,line width=3.9pt,line cap=round] (0,2.0932)--(0,.60);
  \draw[MSPorange,line width=2.9pt,line cap=round] (-.012,2.0932)--(-.012,.60);
  \draw[yellow!55!white,line width=.65pt,line cap=round] (-.030,2.0532)--(-.030,.73);
  \shade[ball color=black] (0,.60) circle[radius=.095];
\node[msp small,anchor=south east] at (-.12,.68) {$O$};
  \node[msp label] at (0,3.13) {$[\mathbb C^5/\mathbb Z_5]$};
  \node[msp small,anchor=west] at (2.63,2.35) {$\nu=0$};
  \node[msp small,anchor=west] at (2.63,-1.65) {$\mu=0$};
  \node[msp small,anchor=east] at (-1.24,-1.57) {$\mathbb P^4$};
  \node[msp small,anchor=west] at (.77,-.56) {$C_Q$};
  \node[msp small] at (0,-1.65) {$Q$};
  \node[msp label] at (0,-2.85) {$K_{\mathbb P^4}$};
  \node[msp small,anchor=west] at (.37,.29) {$x=p=0$};
  \node[msp small,rotate=90] at (.38,1.25) {$\MSPNeckLabel$};

  \node[msp label,anchor=west] at (3.90,2.35) {FJRW theory};
  \node[msp label,anchor=west] at (6.18,2.35) {$\quad \Omega^{\mathrm{LG}}$};
  \node[msp small,anchor=west,align=left] at (3.90,.60) {Cone-vertex \\[-2pt]  theory};
  \node[msp label,anchor=west] at (6.18,.60) {$\quad T\Omega^{\mathrm{pt}}$};
  \node[msp label,anchor=west] at (3.90,-1.65) {GW theory};
  \node[msp label,anchor=west] at (6.18,-1.65) {$\quad \Omega^X$};
}

\newcommand{\id}{\mathrm{I}}

\def\diag{{\mathrm {diag}}}

\def\virt{^{\vir}}

\newcommand{\tw}{\mathrm{tw}}

\def\lsta{_{\ast}}

\newcommand{\CC}{\mathbb{C}}

\newcommand{\PP}{\mathbb{P}}
\newcommand{\QQ}{\mathbb{Q}}

\def\node{{\mathrm{node}}}

\newcommand{\vir}{ {\mathrm{vir}} }
\newcommand{\ev}{ \mathrm{ev} }

\newcommand{\Mbar}{\overline{\mathcal{M}}}

\newcommand{\cal}{\mathcal}

\def\cW{{\cal W}}

\def\ft{\mathfrak{t}}

 \def\sC{{\mathscr C}}

\def\sta{^\ast}

\newtheorem{dummy}{}[section]
\newtheorem{lemma}[dummy]{Lemma}
\newtheorem{proposition}[dummy]{Proposition}
\newtheorem{theorem}{Theorem}

\newtheorem{corollary}[dummy]{Corollary}
\newtheorem{conjecture}[dummy]{Conjecture}

\newtheorem{propositiondefinition}[dummy]{Proposition--Definition}

\theoremstyle{definition}

\newtheorem{definition}[dummy]{Definition}
\newtheorem{example}[dummy]{Example}

\newtheorem{remark}[dummy]{Remark}

\def\bd{\bold d}

\def\loc{{\mathrm{loc}}}
\def\msp{^{\mathrm{msp}}}
\def\pt{\mathrm{pt}}

\def\beq{\begin{equation}}
\def\eeq{\end{equation}}

\def\Pf{{\PP^4}}

\def\barM{{\overline{\mathcal M}}}

\def\sub{\subset}

\def\sH{{\mathscr H}}

\newcommand{\M}{\overline{\mathcal{M}}}

\renewcommand{\ev}{\mathrm{ev}}

\newcommand{\sfq}{q}

\newcommand{\cR}{\mathcal{R}}

\def\lra{\longrightarrow}
\def\and{\quad\mathrm{and}\quad}
\def\bl{\bigl(} \def\br{\bigr)}

\title[Conifold gap for the quintic threefold]{Conifold gap for the quintic threefold}

\author{Huai-Liang Chang}
\address{Institute for Math and AI \& School of Mathematics and Statistics, Wuhan University
}
\email{hlchang@whu.edu.cn}

\author[Shuai Guo]{Shuai Guo}
\address{School of Mathematical Sciences, Peking University
}
\email{guoshuai@math.pku.edu.cn}

\author{Lei You}
\address{School of Mathematical Sciences, Peking University
}
\email{youlei26@stu.pku.edu.cn}

\author{Huaigong Zhang}
\address{Great Bay University, Great Bay Institute for Advanced Study,
}
\email{zhanghuaigong@outlook.com}

\begin{document}

\begin{abstract}


We prove the all-genus conifold gap conjecture for the
Gromov--Witten (GW) theories of the quintic threefold,  other
smooth Fermat Calabi--Yau threefold hypersurfaces,  local
$\mathbb P^2$, and for the FJRW theories associated with these
hypersurfaces.

Our proof uses Mixed Spin P-field (MSP) theory and analogous
master-space constructions. The central object is the
\emph{cone-vertex theory}. 
We express this theory as a translation of the point CohFT via a formal Laplace transform and prove a universal conifold gap theorem for the resulting translated theories.
The master-space construction then transfers the gap from
the cone-vertex theory to the GW and FJRW theories.


\end{abstract}

\setcounter{section}{0}
\setcounter{tocdepth}{1}

\maketitle
\tableofcontents

 \section{Introduction}

 The conifold gap conjecture  \cite{GV,HKQ} predicts that the genus-$g$ GW potential of a Calabi-Yau threefold exhibits  a universal leading pole and vanishing subleading poles,  when the potential is expanded in a flat coordinate near the conifold point. 
 The term ``conifold point'' here refers to  the point of the complex moduli space at
which the corresponding mirror Calabi-Yau fiber acquires an ordinary double
point. The conjecture  supplies boundary conditions for determining the holomorphic ambiguity in the BCOV holomorphic anomaly equations \cite{HKQ} (see also \cite{You} for a survey). For the quintic threefold, the  approach of Guo-Janda-Ruan gives a computational verification for $g\le5$, via the formal quintic and the log GLSM localization formula under development  \cite{GJR2} (see also Lei’s lecture notes \cite{Lei}).  
 For local $\mathbb P^2$,  Brini proved the conjecture in all genera \cite{Brini}. In \cite{FCCZ}, Fang-Chen-Chen-Zong proved a conifold gap theorem for topological recursion. 
 For compact Calabi--Yau threefolds,  the gap conjecture has  remained out of reach in enumerative geometry. One source of difficulty is that the available curve-counting  moduli spaces do not reveal or preserve the gap structure under localization or degeneration. 

 \medskip

 We prove, in all genera, the gap conjecture for the quintic threefold and other smooth Fermat-type Calabi--Yau hypersurfaces.
 The proof uses the same MSP framework as our earlier work on the BCOV Feynman rule for the quintic. That work led us to develop NMSP theory \cite{NMSP1,NMSP2,NMSP3}, where, for large $N$, the BCOV propagators arise from counting NMSP rational curves. Here, however, we return to the original $N=1$ MSP theory and identify a gap phenomenon from its counts of higher genus curves mapped to a neighborhood of the cone vertex lying outside the quintic threefold.

The part of MSP theory used here to study GW invariants, called the \([0,1]\) part, can be viewed as a curve-counting theory on the projective cone over the quintic: $$C_Q=\{[x_1,\cdots,x_5,u], x_1^5+\cdots x_5^5=0\}\sub\PP^5 $$  with vertex $O=[0,0,0,0,0,1]$ and hyperplane section $Q=C_Q \cap \{ u=0 \}$. 
  
   
   The $\CC\sta$-action that scales the last coordinate of the cone    
 has two fixed loci. The first locus, labeled  level zero,   is  the quintic threefold $Q=\{[x_1,\cdots,x_5], x_1^5+\cdots x_5^5=0\} \sub\Pf=(u=0)$. 
The second locus, labeled level one, is the point $O\notin Q$.  

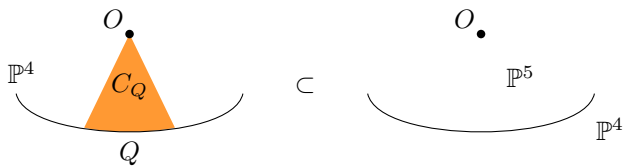
 \begin{center}
\begin{tikzpicture}[x=1cm,y=1cm,text=black,font=\small,
  line cap=round,line join=round]
  \path (-3,-.70) rectangle (2.2,2.08);
  \node[anchor=north west,inner sep=0pt] at (-2.78,1.87)
    {$\mathbb P^5$};

  \path[draw=MSPline!55,fill=MSPblue!24,line width=.55pt]
    (-2.6,-.30)--(1.35,-.30)--(1.95,.35)--(-2.0,.35)--cycle;
  \node[anchor=west,font=\footnotesize,inner sep=0pt]
    at (-2.60,-.52) {$\mathbb P^4=\{u=0\}$};

  \path[fill=MSPorange!18]
    (-.25,1.70)--(-1.27,.02)
    arc[start angle=180,end angle=360,x radius=1.02,y radius=.24]
    --cycle;
  \foreach \ang in {225,270,315}{
    \draw[MSPochre!23,line width=.35pt]
      (-.25,1.70)--({-.25+1.02*cos(\ang)},{.02+.24*sin(\ang)});
  }
  \draw[MSPochre!85,line width=.75pt]
    (-1.27,.02)--(-.25,1.70)--(.77,.02);
  \draw[MSPochre!75,dash pattern=on 2pt off 1.8pt,line width=.6pt]
    (.77,.02) arc[start angle=0,end angle=180,x radius=1.02,y radius=.24];
  \draw[MSPochre,line width=1.05pt]
    (-1.27,.02) arc[start angle=180,end angle=360,x radius=1.02,y radius=.24];
  \node[fill=MSPorange!18,inner sep=1pt] at (-.25,.70) {$C_Q$};

  \draw[black!45,line width=.45pt] (-.16,1.70)--(2.40,1.70);
  \fill[black] (-.25,1.70) circle[radius=1.6pt];
  \node[anchor=west,align=left,inner sep=0pt] at (2.56,1.70)
    {$O$: level $1$\\[-1pt]\footnotesize cone vertex};
  \draw[black!45,line width=.45pt] (.57,-.12)--(2.40,-.12);
  \node[anchor=west,align=left,inner sep=0pt] at (2.56,-.12)
    {$Q$: level $0$\\[-1pt]\footnotesize quintic threefold};
\end{tikzpicture}
\end{center}

 The vertex $O$ is a  conical singularity of $C_Q$. As the  $[0,1]$-MSP theory counts  curves mapped to $C_Q$, torus localization reduces it to its level-zero ($Q$'s GW) and  level-one theory which counts curves mapped to a neighborhood of the cone vertex $O$ in $C_Q$. We call the latter "cone vertex theory". 
 
 This cone singularity lies outside the quintic and appears, at first sight, to be unrelated to the conifold singularity occurring inside the (B side) conifold Calabi--Yau threefold in gap conjecture. However,  in \cite{MSPg1}, the authors suggested that these two singularity pictures are related. More precisely,  they showed that the principal part predicted by genus one gap 
  conjecture  appears  in the  formula counting genus-one   curves near the cone vertex $O$ in MSP theory. 

  We show that, for arbitrary genus, quintic MSP theory then transfers the desired gap property from its cone vertex theory to the level-zero sector---the quintic GW theory. This transfer is made possible by the vanishing of non-leading graph contribution after an $R$-matrix factorization. 
   The same cancellation persists for the MSP theories of local $\PP^1$  (the resolved conifold), local $\PP^2$, and other one-parameter models: one-parameter Calabi-Yau threefolds, as well as FJRW theories for Fermat singularities.

For each of these models, the gap property of its MSP cone-vertex theory follows by specialization from a universal version of the cone-vertex gap for the resolved conifold. We obtain this cone-vertex gap by master-space localization from the resolved-conifold gap, which follows from the multiple-cover formula of Faber and Pandharipande \cite{FP}. 

    

\begin{center}
\begin{tikzpicture}[font=\small,align=center,draw=black,text=black,
  line width=0.4pt]
  \node (rcgap) at (0,1.5) {local $\PP^1$'s \\gap};
  \node (universalgap) at (7.3,1.5) {Universal cone-vertex\\gap};
  \node (quinticgap) at (0,0) {Calabi--Yau threefold\\ $X$'s gap};
  \node (conegap) at (7.3,0) {Cone-vertex gap\\associated with $X$};
  \draw[->] (rcgap.east) -- node[above,font=\scriptsize] {master-space localization}
    node[below,font=\scriptsize] {} (universalgap.west);
  \draw[->] (universalgap.south) -- node[right,font=\scriptsize] {specialization} (conegap.north);
  \draw[->] (conegap.west) -- node[above,font=\scriptsize] {MSP localization} (quinticgap.east);
\end{tikzpicture}
\end{center}
Theorem~\ref{thm:rankone} gives the universal cone-vertex gap,
which specializes to each model through its normalized vanishing period.


\subsection{Conifold gap conjecture for one-parameter Calabi--Yau families} \label{gapconj}
Let $X$ be a smooth projective Calabi--Yau threefold with $h^{1,1}(X)=1$. Suppose that $X$ admits a one-parameter mirror family $\mathcal{X}^{\vee} \to \mathcal{M}$. Here, $\mathcal{M}$ is a rational curve parametrized by a global coordinate $q$. Furthermore, the compactification of $\mathcal{M}$ contains two special points: a maximally unipotent monodromy (MUM) point at $q=0$, corresponding to the large complex structure limit, and a conifold point, where the corresponding fiber acquires ordinary double points.

Put $D=q\partial_q$ and $K=\mathbb Q(q)$. Let $I_0$ and $I_1$ be the holomorphic and logarithmic periods near $q=0$, normalized by $I_0=1+O(q)$ and $I_1=I_0\log q+O(q)$; {they satisfy} a common homogeneous linear Picard-Fuchs equation over $K$. Set $\tau=I_1/I_0$, the flat coordinate.

\begin{definition}
Let $\mathscr R$ be the differential $K$-algebra generated by $I_0^{\pm1}$ and $(D\tau)^{\pm1}$.
\end{definition}
The Picard-Fuchs equation  implies that $\mathscr R$ is finitely generated as a $K$-algebra.

\medskip

Choose a local parameter $\Delta\in K$ at the conifold point and expand elements of $K$ in $\mathbb C((\Delta))$. Then $D=(D\Delta)\partial_\Delta$.
Let $\omega_0$ and $\omega_1$ be holomorphic periods such that $\omega_0(0)\neq 0$, $\omega_1$ vanishes simply at $\Delta=0$, and the logarithmic period has the form $\omega_1\log\Delta+O(\Delta^0)$.
Define the conifold mirror map by $\tau_c=\omega_1/\omega_0$.

We call a pair $(\omega_0,\omega_1)$ \emph{admissible} if it satisfies the above period conditions and if the substitutions $I_0\mapsto\omega_0$ and $D\tau\mapsto D\tau_c$, extended to derivatives and inverses, preserve the relations in $\mathscr R$. They then define a differential homomorphism:
\[
\iota:\mathscr R\longrightarrow
\mathscr R^{\mathrm{coni}}\subseteq\mathbb C((\Delta)),
\]
where $\mathscr R^{\mathrm{coni}}$ is the ring generated by $\omega_0^{\pm1}$ and $(D\tau_c)^{\pm1}$ and their derivatives.

\begin{definition}
For $f\in\mathscr R$,  we call the Laurent series $f_c:=\iota(f)$  its \emph{conifold expansion}, given admissible periods $\omega_0,\omega_1$,  and call $f \in \mathscr R$ \emph{conifold regular} if $\iota(f)\in\mathbb C[[\Delta]]$.
A matrix or vector is said to be conifold regular if all of its entries  lie in   $\mathscr R$ and are conifold regular.
\end{definition}



\begin{example}
For the quintic, take $\Delta=1-5^5q$. The Picard–Fuchs equation becomes
\[
\Big(
D^4+(\Delta-1)\prod_{j=1}^4
\big(D+\frac j5\big)
\Big)\omega=0,
\qquad
D=(\Delta-1)\partial_\Delta.\]
 At the MUM point $q=0$, normalize $I_0=1+O(q)$ and $I_1=I_0\log q+O(q)$. Put
$I_{11}=D\tau$,
$A_1=\frac{DI_{11}}{I_{11}}$ and
$B_k=\frac{D^kI_0}{I_0}$ for $k=1,2,3$. The differential ring is given by
\[
\mathscr R
=
K[I_0^{\pm1},I_{11}^{\pm1},A_1,B_1,B_2,B_3].
\]
 At the conifold point  $q=5^{-5}$, an admissible pair of periods can be chosen as
\[
\omega_0=1+O(\Delta^3),
\qquad
\omega_1=\Delta+\frac7{10}\Delta^2+O(\Delta^3).
\]
These conditions determine both solutions in $\mathbb Q[[\Delta]]$.
Thus the prescribed substitution defines the differential homomorphism $\iota$, with image in $\mathbb Q((\Delta))$. Set\[
\tau_c=\frac{\omega_1}{\omega_0}
=\Delta+\frac{7}{10}\Delta^2+O(\Delta^3).
\]
\begin{remark}
The choice of $\omega_0$ is not unique. Let
$\omega_2=\Delta^2+O(\Delta^3)$ be the holomorphic Picard--Fuchs
solution with this leading term. For example, keeping $\omega_1$ fixed, we may take
\[
\omega_0'=\omega_0+c_1\omega_1+c_2\omega_2,
\qquad c_1,c_2\in\mathbb Q.
\]
The substituted generators satisfy the same Picard--Fuchs and
Wronskian differential identities. Thus their finite ring remains
closed under $D$, and $(\omega_0',\omega_1)$ is again admissible.
The proof below applies to each such choice, with $\iota$, $\tau_c$
and all conifold expansions defined using the chosen pair.
\end{remark}
\end{example}



Let $F_g(Q)$ be the genus-$g$ Gromov-Witten potential of the quintic. The finite generation theorem proved in \cite{NMSP2} states that, for $g\ge2$ 
$$F_g(Q(q)) \in \mathscr R ,  \qquad  \text{where} \quad Q(q)=\exp(I_1/I_0).$$
Write $F_g^c(\tau_c)$ for $\iota(F_g(Q(q)))$, expressed in the local coordinate $\tau_c$.

\begin{conjecture}[Quintic conifold gap] For any admissible pair $(\omega_0,\omega_1)$ and any $g\ge2$,
\[
F_g^c(\tau_c)
-
\frac{(-1)^{g-1}B_{2g}}
{5^{g-1}\,2g(2g-2)\,\tau_c^{2g-2}}
\in\mathbb C[[\tau_c]],
\]
 where $B_{2g}$ is the Bernoulli number. Thus the displayed leading pole is the only negative-power term.
\end{conjecture}

\begin{remark}
With the $R$-matrix and shifts constructed below, the master-space reconstruction defines the
local CohFT over $\mathscr R$. Integration, with the prescribed shifts
and dilaton resummation, therefore determines $F_g$ canonically as an
element of $\mathscr R$ for $g\ge2$. Since $\mathscr R$ is finitely
generated, this also proves the finite-generation statement
$F_g\in\mathscr R$. No algebraic independence is required: the
canonical object is an element of $\mathscr R$, not a uniquely chosen
polynomial expression in its generators.  

For the quintic, algebraic independence of the relevant period generators is known; see \cite[Theorem~2]{Mov11} and \cite[Theorem~2.14 and its proof]{ZZ25}.
\end{remark}

\subsection{Master-space localization as an $R$-matrix action}

For the quintic Calabi-Yau threefold $Q$, or more generally for a smooth Calabi-Yau complete intersection threefold $X$ in a weighted projective space (or a product thereof), the master-space technique provides a framework for computing its GW invariants. 

The idea is to construct a larger theory that behaves like the GW theory of a Fano variety, with a torus action whose fixed-locus contributions include those associated with $X$ and a point. Indeed, the GW theory of $C_Q$ discussed in the  introduction can be defined through the MSP construction; it is exactly the $[0,1]$ sector of quintic MSP theory. Localization then yields relations among these invariants.



The even\footnote{For simplicity, we suppress the odd part of the state space in the notation throughout the rest of the paper.  All CohFTs are understood on the full graded state spaces. The odd sector in the unmarked graph sum is treated in the proof of Theorem 1.}  state space  $\sH^X=H^{\mathrm{even}}(X)$ admits an ordered homogeneous basis
\[
\{ \varphi_0=\mathbf1_X,\ \varphi_1,\ \varphi_2,\  \varphi_3 \} ,
\qquad \deg\varphi_i=i.
\]
Here $\deg$ denotes complex cohomological degree. 
We write $\{\varphi^i\}$ for the basis
dual to $\{\varphi_i\}$ with respect to the state-space pairing.
We use column vectors throughout.

Suppose that we have a master-space construction as above for which virtual localization, together with Grothendieck-Riemann-Roch, yields the CohFT relation \footnote{We refer to \cite{PPZ,NMSP3} for the definitions of CohFTs, the $R$-and $T$-action.}
\begin{equation}\label{eq:master-cohft-R}
 \Omega^M=R T_R(\Omega^{X,\tau}\oplus\Omega^\pt)
\end{equation}
where $\Omega^M$ is the CohFT associated with the master theory, $R$ is the symplectic $R$-matrix obtained from this construction and $T_R(z)
=
z\left(
\mathbf1_X+\mathbf1_{\mathrm{pt}}
-
R(z)^{-1}\mathbf1_M
\right)$ is the translation.

Using the comparison of fundamental solutions supplied by the construction, the QDEs for the master theory and the fixed-locus theories give the following relative QDE,
\begin{equation}\label{eq:two-QDE}
-zD R(z)^{-1}+A^{\mathrm{loc}}R(z)^{-1}=R(z)^{-1}A^M,
\qquad A^{\mathrm{loc}}=A^X\oplus A^{\mathrm{pt}},
\end{equation}
where $A^M$, $A^X$, and $A^{\mathrm{pt}}$ are the corresponding quantum product matrices.

{
 \begin{definition} \label{def:globallocal}
 Fix the ring $\mathscr R$ and a conifold expansion homomorphism $\iota$.
The relative QDE \eqref{eq:two-QDE}, together with the symplectic $R$-matrix relating them, whose entries take values in $\mathscr R$, are said to form a \emph{regular local-global QDE system} if the following conditions hold:
\begin{enumerate}
\item All entries of $A^M$ and $A^{\mathrm{pt}}$ take values   in the ring $\mathscr R$ and are conifold regular.
\item   
The inverse of the matrix with columns $(A^M)^j\mathbf1_M$, $0\leq j \leq 4$, is conifold regular. 
The entry $\eta^X(\varphi^1,A^X\varphi_0)$ and its inverse are conifold regular.\footnote{For the quintic, in the basis used below, this coefficient is $I_{11}$, with conifold expansion $D\tau_c=-1+O(\Delta)$.}
\item The $\varphi_0$- and $\varphi_1$-components of $R(z)^{-1}\mathbf1_M$ are conifold regular, with the former equal to $I_0$ at $z=0$. 
\end{enumerate}
\end{definition}
}


\subsection{Abstract conifold gap theorem}

\begin{definition}\label{def:conifold-reg-gap}
Let  $\mathscr H$  be a fixed finite-dimensional vector space, and let
 $\Omega$  be a CohFT with coefficients in $\mathscr R$ and state space $\mathscr H$. 

\begin{enumerate}
\item
The CohFT $\Omega$ is \emph{conifold regular} if
for $2g-2+n>0$, $v_i\in \mathscr H$, and $a_i\geq 0$, the correlator is conifold regular:
\[
  \iota\bigg(
  \int_{\overline{\mathcal M}_{g,n}}
  \Omega_{g,n}(v_1,\ldots,v_n)
  \prod_{i=1}^n\psi_i^{a_i}\bigg)
  \in\mathbb{C}[[\Delta]].
\]

\item
Fix a formal series
$\omega=\Delta+O(\Delta^2)\in \mathbb{C}[[\Delta]]$.
The CohFT $\Omega$ has the \emph{conifold gap property}
with respect to  $\omega$  if   there exists a constant $\kappa\in\mathbb C^*$ such that, for every  $g\geq2$ 
\[
  \iota\bigg(
  \int_{\overline{\mathcal M}_{g}}
  \Omega_{g,0}
  \bigg)
  =
  \kappa^{1-g}\frac{(-1)^gB_{2g}}{2g(2g-2)}\,\omega^{2-2g}+O(\Delta^0) .
\]
\end{enumerate}
\end{definition}

 \begin{theorem} \label{thm:main}
Suppose the following three conditions hold.
\begin{itemize}
\item[(i)] There is a conifold-regular theory $\Omega^M$ over $\mathscr R$, with
state space $\sH^M$, pairing $\eta^M$ and unit $\mathbf 1_M$, and a symplectic
$R(z):\sH^X\oplus \mathbb C \mathbf1_{\mathrm{pt}}\to \sH^M$  such that
\begin{equation}\label{eq:master-decomp}
 \Omega^M=RT_R(\Omega^{X,\tau}\oplus\Omega^\pt).
 \end{equation}
 \item[(ii)]  The above $R$-matrix is determined by a \emph{regular local-global QDE system}.
   \item[(iii)] The translated point theory $T\Omega^{\pt}$, where $T=T_R|_{\mathrm{pt}}$, has  the \emph{conifold gap property} with respect to $\omega_1$.
\end{itemize}
Then the 
$F_g^{X,c}$ below has the following \emph{conifold gap} property:
$$
F_g^{X,c} := \iota\Big( \int_{\overline{\mathcal M}_g} \Omega_g^{X,\tau} \Big)
= \kappa^{1-g}\frac{(-1)^{g-1} B_{2g}}{2g(2g-2)\,\tau_c^{2g-2}}+O(\tau_c^0)
$$
where $\tau_c = \omega_1/\omega_0 \in\Delta\mathbb C[[\Delta]]$ is the conifold mirror map.
\end{theorem}

 \begin{theorem} \label{main2}
Let $X=Q$ be the quintic threefold, and let
$\Omega^M=\Omega^{[0,1]}$ be the CohFT of the ordinary $N=1$ MSP
$[0,1]$-theory. For every $g\ge2$, the quintic Gromov-Witten potential satisfies
\[
F_g(Q(q))\in\mathscr R.
\]
With $\omega_1$ normalized as in Section~\ref{gapconj} and any
admissible choice of $\omega_0$, the associated MSP construction
satisfies Conditions (i)--(iii) of Theorem~\ref{thm:main}.
Consequently, the conifold gap property holds for the quintic threefold.
\end{theorem}
For the quintic threefold, Condition (i) follows from the geometric degree bound for the MSP $[0,1]$-theory established in Section~\ref{sec:msp}.
Condition (ii) is verified by the genus-zero calculations in Section~\ref{QDEsystem}.  
Condition (iii) follows from the gap theorem for the cone-vertex theory in Section~\ref{sec:TKW}. The proof of {Theorem \ref{main2} } is completed in  Section~\ref{proofmain2}.
 
The same construction gives finite generation in the corresponding
period rings and the conifold gap for the one-parameter Fermat
Calabi--Yau hypersurfaces of degrees $6$, $8$, and $10$ in
Section~\ref{sec:uniform-genus-zero}. Both conclusions also hold for
local $\mathbb P^2$ in Section~\ref{sec:local-p2-application} and for the FJRW theories
associated with the four Fermat Calabi--Yau hypersurfaces in
Section~\ref{sec:fjrw-application}. Section~\ref{sec:uniform-genus-zero}
also explains how both conclusions extend to other one-parameter
models once the required MSP comparisons and polynomiality are established.

 \medskip

\begin{remark}
Guo, Janda, and Ruan \cite{GJR2} suggested an alternative approach to the quintic conifold gap conjecture via the formal quintic, verifying the conjecture up to genus five. Their localization formula under development reduces the  conjecture to two key statements: the conifold gap for the formal quintic, and the conifold regularity of its correlators with special insertions. While we expect our method to prove the conifold gap for the formal quintic as well, a corresponding formal theory may not exist for general Calabi--Yau threefolds. Consequently, their approach cannot be directly extended to one-parameter models for which no such formal counterpart is available. 

\end{remark}

 \medskip

The rest of the paper is organized as follows. Section~\ref{sec:msp} reviews MSP theory and establishes the properties needed for the proof. Section~\ref{sec:decompR} decomposes the $R$-matrix and proves the abstract gap theorem.
Section~\ref{sec:TKW} establishes the gap property for the cone-vertex theory. Section~\ref{sec:applications} discusses applications to other one-parameter models. The appendix provides the calculation of the Dijkgraaf--Witten coordinate.\\


\noindent {\bf Acknowledgements.}
The authors thank Jun Li, Weiping Li, Chiu-Chu Melissa Liu,
Yongbin Ruan, Emanuel Scheidegger and Yang Zhou for 
collaborations on related projects  and for valuable discussions.
   This project began in 2024, when H.-
L. Chang, L. You, and H.  Zhang were based at HKUST, where substantial progress was achieved. We are grateful to HKUST for providing a highly supportive   environment for our mathematical discussions. \\

 \noindent\textbf{AI disclosure.}
The authors used ChatGPT to assist with language editing,
figure preparation, and the improvement of Maple programs previously
developed by the authors.   All mathematical results and proofs are due
to the authors, who take full responsibility for the accuracy and
integrity of this work. 

 $$$$

 \section{Mixed Spin P-field revisited} \label{sec:msp}

As mentioned in the introduction, we use the master space technique to obtain an algorithm for computing GW invariants. For the quintic Calabi-Yau threefold, we consider mixed spin $P$-field (MSP) theory \cite{CLLL}. Its generalization, NMSP theory, was systematically studied in \cite{NMSP1,NMSP2,NMSP3}. Here we use the original MSP theory for simplicity. Although its degree bound is weaker than that of NMSP theory, it suffices for our proof of the gap conjecture. 

The main result of this section is the following theorem, which is entirely parallel to the corresponding result in NMSP theory. \footnote{The material in this section is based on an unpublished note written in 2017 by the first two authors and Jun Li, before the development of NMSP theory. The main computational result was already recorded in the footnote on page 623 of \cite{NMSP2}. 
}


\begin{theorem}\label{Omega01}
The MSP theory admits a sub-theory, called its $[0,1]$-part, which defines a CohFT $\Omega^M\!$ with the following properties.
\begin{itemize}
\item
The state space  is 
$\mathscr H^M  \oplus H^{3}(Q)$,  where $\mathscr H^M$ is
the image of the restriction map
\begin{equation*}
i^*:H_{\mathbb C^*}^{\mathrm{even}}(\mathbb P^5)
\longrightarrow H_{\mathbb C^*}^{\mathrm{even}}(Q\sqcup\pt),
\qquad
i:Q\sqcup\pt\hookrightarrow\mathbb P^5. 
\end{equation*}
Its kernel is $
\{\alpha\in H_{\mathbb C^*}^{\mathrm{even}}(\mathbb P^5):
5p\cup\alpha=0\}$.
Here $p$ is the hyperplane class and ${\mathbb C^*}$ acts on the last homogeneous coordinate. 
The pairing is induced by
\begin{equation*}
(a,b):=\int_{\mathbb P^5}a\cup b\cup e(\mathcal O_{\mathbb P^5}(5)).
\end{equation*}
The restrictions of $1,p,\ldots,p^4$, denoted by the same symbols, form a basis of $\mathscr H^M$.
\item
In genus zero, the $[0,1]$-correlators with even insertions coincide
with those of the $\mathcal O(5)$-twisted
$\mathbb C^*$-equivariant Gromov--Witten theory of $\mathbb P^5$.
\item  The $[0,1]$-CohFT $\Omega^{M}\!$ can be  realized as an $R$-matrix action on the local CohFT
\begin{align}\label{Omega01cohft}
\Omega^{M} = R T_R(\Omega^{Q,\tau_Q} \oplus  \Omega^\pt ),
\end{align}
where the entries of the $R$-matrix lie in the ring $\mathscr R$.
  Here $\Omega^{Q,\tau_Q}$ is the CohFT of quintic's GW shift by coordinate $\tau_Q$, and 
 $ \Omega^\pt$ is  the GW of a point.  
 \item  For $2g-2+n>0$ and $m_i\in [0,4]$, the $[0,1]$-correlator
\begin{align}\label{01des-gen}
\int_{\M_{g,n}}  \psi_1^{k_1} \cdots \psi_n^{k_n} \, \Omega^M_{g,n}( p^{m_1} ,\cdots,   p^{m_n}  )   \in  \QQ(\ft)[q]
\end{align}
is a polynomial of $q$-degree no more than
$4g-4+  \sum_{i=1}^n  m_i $. Here $\ft$ is the equivariant parameter of the $\mathbb C^*$ action.
 \end{itemize}
 \end{theorem}


 We briefly recall the development of MSP and NMSP  theory. Prior to the gap conjecture in \cite{HKQ}, the work of 
 Bershadsky-Cecotti-Ooguri-Vafa (\cite{BCOV}) and Yamaguchi-Yau (\cite{YY})
on the $B$ side led to  a series of   conjectures concerning all genus Gromov-Witten theory of  Calabi-Yau threefolds.  For the quintic, mixed-spin $P$-field theory interpolates between the $P$-field theory of the Gromov-Witten phase \cite{CL} and the spin-field theory of the FJRW phase \cite{CLL15} by promoting the Kähler parameter in Witten's GLSM to a field $\mu$. Its generalization with $N$ such fields is called NMSP theory.   For sufficiently large $N$, NMSP theory provides geometric explanations for the BCOV/Yamaguchi-Yau structures, including the BCOV Feynman rule with the $3g-3$ degree bound, the holomorphic anomaly equation, Yamaguchi-Yau polynomiality, and the analyticity of $F_g$; see \cite{NMSP2,NMSP3}. These make the quintic the first compact Calabi-Yau threefold for which all these predictions were established.  The MSP and NMSP moduli spaces have also been generalized, via $\Omega$-stability, to complete intersection Calabi-Yau threefolds in toric varieties \cite{CGLLZ}, and proofs of the corresponding predictions are in progress.  The construction of \cite{CGLLZ} also yields BCOV-type higher-genus structures for Landau-Ginzburg theories beyond Gromov-Witten theory, including FJRW and hybrid theories \cite{Z}. 

 \subsection{The moduli space}

An MSP field $\xi$ consists of a pointed twisted curve $(\sC,\Sigma)$, two representable line bundles $\sL$ and $\sN$, and fields:
\beq\label{xi0}
\xi=(\sC,\Sigma,\sL,\sN,\phi=(\phi_1,\cdots,\phi_5),\rho,\mu,\nu),
\eeq
where
$\phi_i\in H^0(\sL)$, $\rho\in H^0(\sL^{-5}\otimes\omega_{\sC}^{\log})$, and $
(\mu,\nu)\in H^0(\sL\otimes\sN\oplus\sN)$,
satisfy, among other things (cf. \cite{CLLL}), that all $(\phi,\mu)$, $(\rho,\nu), (\mu,\nu)$ are nowhere vanishing.  We require that $\rho$ vanishes on the scheme markings $\Sigma=\{\Sigma_\ell\}$ (called $1_\rho$ markings).  If $\text{Aut}\,\xi$ is finite, we call $\xi$ stable. 
\\

We denote by $\cW_{g,n,\bd}$ the DM stack of the moduli of stable MSP fields of genus $g$, 
$n$-markings of type $1_\rho$, and with $\bd= (d_0,d_\infty)$ where $d_0=\deg \sL\otimes\sN$ and $d_\infty=\deg\sN$. Further, $\cW_{g,n,\bd}$ is a $G=\CC\sta$ stack, where $G$ acts on $\xi\in \cW_{g,n,\bd}$ by scaling its field $\mu$. We denote its cosection localized ($G$-equivariant) virtual cycle by $[\cW_{g,n,\bd}]\virt$. 
For Gromov-Witten theory we only need $d_\infty=0$.\\

 At the $\ell$-th marking, the formula $\ev_\ell(\xi)=[\phi_1|_{\Sigma_\ell},\cdots, \phi_5|_{\Sigma_\ell}, \mu/\nu|_{\Sigma_\ell} ] \in \sL|_{\Sigma_\ell}^{\oplus 6}$ defines a $G$-equivariant evaluation map $$\ev_\ell: \cW_{g,n,\bd}\lra \PP^5,$$
  where $G$ acts on $\PP^5$ by scaling the last coordinate. 
Denote by $p$ the equivariant hyperplane class in $H^2_G(\PP^5)$. 
Let $\pt=[0,0,0,0,0,1]\in (\PP^5)^G$ be the isolated fixed point, and regard $Q\sub\PP^4\sub (\PP^5)^G$ via its tautological embedding. Then
$$p|_{Q}=H\in H^2(Q) \and p|_{\pt}:=-\ft .
$$
By sending $p^i\in H^{2i}_G(\PP^5)$ to $(p^i|_{Q},p^i|_{\pt})$, 
we get the 
linear map
\beq\label{Harrow}
H\sta_G(\PP^5) = \QQ(\ft) [p]/\bl p^5(p+\ft)\br\lra   \sH :=  H^{\mathrm{even}}(Q)\oplus H^0(\pt).
\eeq
 
Further, the kernel of the restriction \eqref{Harrow} is spanned by $p^4(p+\ft)$. 
 The integral of the pullback of any such class against $[\cW_{g,n,\bd}]\virt$  vanishes by the torus localization formula. We therefore obtain the total MSP theory:  for $\tt_1(z),\cdots,\tt_n(z)\in \sH[z],$
\begin{equation}\label{mspsum}
 \Bigl<  \prod_{i=1}^n \tt_i(\psi)\Bigr>\msp_{g,n} 
=  \sum_{d\geq 0} (-1)^{d+1-g } q^d \int_{[{\mathcal W}_{g,n,(d,0)}]\virt} \prod_{i=1}^n  \tt_i(\psi_i).
\end{equation}  
 
 The torus action gives rise to localization graphs. Each graph has vertices at level $0$ (the quintic Gromov-Witten sector, characterized by $\mu=0$), level $1$ (the point sector, characterized by $\phi=\rho=0$), or level $\infty$ (the FJRW sector, characterized by $\nu=0$), joined by edges. A graph containing an edge from level $0$ to level $\infty$ is called irregular and contributes zero to the localization formula \cite{van}. Hence we may restrict to regular graphs and decompose the localization sum into a $[0,1]$ part and a $(1,\infty]$ part. For details, see \cite{CLLL}.\\

 We say a graph is supported on $[0,1]$ if all its vertices lie at level $0$ or level $1$.  Let $\gndzo$ be the set of localization graphs  that are supported on $[0,1]$.  The $[0,1]$ theory is defined as the sum of the localization contributions  of these graphs:
$
[\cW_{g,n,(d,0)}]^{[0,1]}:= \sum_{\locg \in  \gndzo}  \frac{[F_\Theta]\virt}
{e(N\virt_{\Theta})}
$.

\begin{definition}\label{01theory-def}
For any $2g-2+n>0$ and $\tt_i(z) \in \sH[\![z]\!]$, $i=1,\cdots,n$,  define
$$
\bigl[\otimes_i^n \tau_{i}(\bar\psi_i)\bigr]^{[0,1]}_{g,n}
=  \sum_{d\ge 0}  (-1)^{d+1-g } q^d 
\bl    \text{pr}^W_{g,n}   \br\lsta
 \Bigl(
 \prod_{i=1}^n \ev_i\sta  \tt_i(\bar\psi_i)\cdot [\cW_{g,n,(d,0)}]^{[0,1]} 
\Bigr) 
$$
which lies in $ H^*(\M_{g,n},\QQ(\ft)[[q]])$. 
Here $\pr^W_{g,n}:\cW_{g,n,(d,0)}\to\M_{g,n}$ denotes the projection. 
For constant insertions $\tt_i(z)=t_i$,  these classes form
a CohFT as $(g,n)$ ranges over all stable pairs.  We denote it by $\Omega^M$ and call it the  MSP-$[0,1]$ CohFT.

\end{definition}

\subsection{Proof of the Theorem~\ref{Omega01}}

The assertions on the state space and genus-zero theory follow directly from the preceding geometric construction, as in \cite[Section~1]{NMSP2}. 

For the CohFT formula, the virtual localization argument of \cite[Section~3]{NMSP2} gives
$$
\Omega^M = R^{\mathrm{loc}} T_{R^{\mathrm{loc}}}(\Omega^{Q,\tw,\tau_{Q,\tw}} \oplus \Omega^{\pt,\tw}) .
$$
The localization $R$-matrix   is characterized by the Birkhoff factorization 
$$
S^M(z) = R^{\loc}(z) \cdot  \diag( S_{\tau_{Q,\tw}}^{Q,\tw}(z), S_{\tau_{\pt}}^{\pt,\tw}(z))   , 
$$
where $\tau_{Q,\mathrm{tw}}$ and $\tau_{\mathrm{pt}}$ are the Dijkgraaf-Witten coordinates (see \cite[Section~1.4]{NMSP2}).

Applying Grothendieck-Riemann-Roch to the twisting classes gives the quantum Riemann-Roch operators $\Gamma^Q(z)$ and $\Gamma^{\mathrm{pt}}(z)$, including the constant factors identifying the untwisted and twisted pairings. 
These operators map the untwisted Lagrangian cones to the twisted ones, as in \cite[Corollary~4]{CG}.  Hence there exist a unique matrix $R^{\tw} = R_0^{\tw}+R_1^{\tw} z+R_2^{\tw} z^2+\cdots  $  and the coordinate $\tau_Q$, such that
 $$
S_{\tau_{Q,\tw}}^{Q,\tw}(z) \Gamma^Q(z)  = R^{\tw}(z) S_{\tau_{Q}}^Q(z) .
 $$
 Combining quantum Riemann-Roch with Givental’s reformulation of the Kontsevich-Manin ancestor-descendant relation \cite[Theorem~1 and Appendix~2]{CG}, we obtain, using the above Birkhoff factorization,
$$
\Omega^{Q,\tw,\tau_{Q,\tw}}  = R^{\tw} T_{R^{\tw}}\Omega^{Q,\tau_{Q}}. 
$$
For the point component,
\[
S_{\tau_{\mathrm{pt}}}^{\mathrm{pt},\mathrm{tw}}(z)
=
S_{\tau_{\mathrm{pt}}}^{\mathrm{pt}}(z)
=
e^{\tau_{\mathrm{pt}}/z},
\]
 so this factor commutes with the scalar operator $\Gamma^{\mathrm{pt}}(z)$.
 Composing the two actions therefore gives
$$
R(z) = R^{\loc}(z) \cdot \diag (R^{\tw}(z), \Gamma^{\pt}(z))
$$
and the CohFT relations
$$
\Omega^M = R T_R(\Omega^{Q,\tau_Q} \oplus \Omega^{\pt}). 
$$ 
 This proves \eqref{Omega01cohft}.
Moreover, the resulting $R(z)$ is characterized by  
 \begin{equation} \label{birk}
S^M(z)\cdot   \begin{pmatrix}  \Gamma^Q(z)  & \\ &  \Gamma^{\pt}(z)   \end{pmatrix}      = R(z)  \begin{pmatrix}  S^Q_{\tau_Q}(z) & \\ & e^{\tau_{\pt}/z}  \end{pmatrix} .
\end{equation}

Finally, we prove the degree bound. Since
$\operatorname{vdim}\mathcal W_{g,n,(d,0)}=d+1-g+n$,
a nonzero coefficient of $q^d$ in the total MSP correlator with ancestor insertions $p^{m_i}\bar\psi_i^{k_i}$ must satisfy
\[
d+1-g+n\leq\sum_{i=1}^n(m_i+k_i),
\qquad
\sum_{i=1}^n k_i\leq3g-3+n.
\] 
 Thus the total ancestor correlator is a polynomial in $q$ of degree at most $4g-4+\sum_i m_i$. Applying the bipartite-graph argument of \cite[Section~4]{NMSP2} directly to the ordinary MSP theory, induction on the genus gives the same bound for the contributions of bipartite graphs containing level-$\infty$ vertices. Subtracting their sum from the total correlator proves the asserted degree bound for the $[0,1]$-theory.

 
  This proves Theorem~\ref{Omega01} except for the assertion that the coefficients of $R(z)$ lie in $\mathscr R$, which is established in the next subsection.

\subsection{Properties of  $A^M$, $A^Q$ and the $R$-matrix} \label{QDEsystem}

 In this subsection, 
we prove the assertion in Theorem~\ref{Omega01} that the entries of $R(z)$ lie in $\mathscr R[[z]]$ for the quintic case. \footnote{Here $\mathscr R$ is understood with $q$ replaced by $q/\mathfrak t$; alternatively, we may set $\mathfrak t=1$ at the end.} We also explicitly compute the QDE matrices $A^M$, $A^Q$ and $A^{\mathrm{pt}}$, together with $R(z)^{-1}\mathbf1_M$. 

We will use the bases $\{\mathbf 1_M,p,\ldots,p^4\}$ on the
master space,  $\{\mathbf 1_Q,H,H^2,H^3\}$ on the quintic and   $\{ \mathbf1_{\mathrm{pt}}\}$ on the point.

{

\subsubsection{Computation of $A^M$ and the $S^M$-matrix.}
We compute $A^M$ and $S^M$ by Birkhoff factorization, following \cite[Section~1.5]{NMSP2}. The mirror theorem gives
\begin{equation}\label{msp-unit}
S^M(z)^*\mathbf1_M=\frac{J^M(0,z)}z
=e^{-120q/z}\sum_{d\ge0}q^d
\frac{\prod_{m=1}^{5d}(5p+mz)}
{\prod_{m=1}^{d}(p+mz)^5(p+\ft+mz)},
\end{equation}
where $p^4(p+\ft)=0$. 
Put $D_p=p+zD$. The QDE $$D_pS^M(z)^*=S^M(z)^* A^M$$ and the property $S^M(z)=\id+O(z^{-1})$ imply  that
\begin{equation}\label{msp-A}
S^M(z)^*p^{j+1}
=D_p\bigl(S^M(z)^*p^j\bigr)
-\sum_{i=0}^j(A^M)_{ij}S^M(z)^*p^i,\ 0\le j<4
\end{equation}
is of the form $p^{j+1}+O(z^{-1})$.
Starting from \eqref{msp-unit}, these formulas determine $S^M(z)^*$ recursively and all entries of $A^M$.

The coefficient of $q^dz^{-k}$ in $S^M(z)^*p^j$ is a polynomial in $p,\ft$ of degree $k-d+j$. Hence $\deg_q(A^M)_{ij}\le j+1-i$, and so the terms through $q^5$ suffice to compute $A^M$. In the ordered basis $1,p,\ldots,p^4$, the recursion gives
\[
A^M=
\begin{pmatrix}
0&*&*&*&*\\
1&*&*&*&*\\
0&1&*&*&*\\
0&0&1&*&*\\
0&0&0&1&*
\end{pmatrix},
\]
 where each starred entry belongs to $\mathbb Q[q,\ft]$. The matrix with columns $(A^M)^j\mathbf1_M$, $0\le j\le4$, is therefore upper unitriangular and has an inverse over $\mathbb Q[q,\ft]$.

\subsubsection{The quintic coordinate $\tau_Q$ and quintic component of $R(z)^{-1}\mathbf 1_M$.} \label{quintictauq}
We consider the   $\tau_Q$  defined in \eqref{birk}. A footnote on   p.~623 of \cite{NMSP2} states that
$$\tau_Q=\tau_0+\tau_1H+\tau_2H^2+\tau_3H^3$$ where (see a detailed calculation in Appendix~\ref{app:dwmap})
\begin{equation}\label{msp-tau} \small
\tau_0=-120q,\quad \tau_1= \frac{I_1}{I_0}-\log q,\quad
\tau_2=-\frac{Y}{2\ft I_0^2I_{11}},  \quad
\tau_3=\frac{Y}{\ft^2 I_0^2}
\left(-\frac1{12}+\frac Y4-\frac{B_1}2-\frac{A_1}4\right).
\end{equation}
The same formula holds for other smooth Calabi--Yau hypersurfaces,
with only the term $\tau_0$ modified accordingly.

Put $I_{22}=Y/(I_0^2I_{11}^2)$ and $I_{33}=I_{11}$. Since $H* H=(I_{22}/I_{11})H^2$ and multiplication by $H^2,H^3$ is classical \cite[Appendix~A]{NMSP2}, quantum multiplication by $H+D\tau_Q$ has matrix
\begin{equation}\label{msp-loc}
A^Q=-120q\,\id+
\begin{pmatrix}
0&0&0&0\\
I_{11}&0&0&0\\
D\tau_2&I_{22}&0&0\\
D\tau_3&D\tau_2&I_{11}&0
\end{pmatrix}.
\end{equation}

\medskip

Now we compute the quintic component of $R(z)^{-1}\mathbf1_M $.
The Birkhoff factorization \eqref{birk} and the MSP mirror theorem \eqref{msp-unit} give
\[
\left.R(z)^{-1}\mathbf1_M\right|_Q
=
S^Q_{\tau_Q}(z) \cdot \Gamma^Q(z)^{-1}   \cdot  (-z)^{-1}e^{120q/z}{\left.I^M(q,-z)\right|_Q}  .
\]
Following the proof of Theorem 4.6 in \cite{CCIT2}, we combine the quantum Riemann-Roch factor with the hypergeometric modification using the $G_0$-operator. This gives
\begin{equation}  \label{Goperator}
\Gamma^Q(z)^{-1}\left.I^M(q,-z)\right|_Q
=
q^{H/z} e^{-G_0(-zD,z)}  
 I^Q(q,-z).
\end{equation}
Here $I^Q(q,z)$ denotes the ordinary quintic $I$-function
and
$
G_0(x,z)
=
\sum_{r,l\ge0}
s_{r+l-1}\frac{B_r}{r!}\frac{x^l}{l!}z^{r-1}$ with  $s_{-1}=0$, $s_0=-\log\ft$, 
$s_k=\frac{(-1)^k(k-1)!}{\ft^k}$ for $k\geq 1$.

The differential operator on the right hand side of \eqref{Goperator} has the expansion
\[
e^{-G_0(-zD,z)}
=
\ft^{-D-\frac12}
\exp\left(
\sum_{\ell\ge1}
\frac{B_{\ell+1}(D+1)}{\ell(\ell+1)}
\left(\frac z{\ft}\right)^\ell
\right)
=:
\ft^{-D-\frac12}
\sum_{m\ge0}
\left(\frac z{\ft}\right)^m
\mathfrak D_m(D).
\]
Here $B_\ell(x)$ is the Bernoulli polynomial. This defines $\mathfrak D_m(D)\in\mathbb Q[D]$ with $\deg_D\mathfrak D_m\le2m$. Since $\ft^{-D}f(q)=f(q/\ft)$, the operator rescales the $q$ in  the quintic $I$-function to $q/\ft$.

Recall the ordinary quintic mirror theorem gives
\[
S^Q_{\tau_1H}(z)\, q^{H/z}
 {I^Q(q/\ft,-z)}
={-z} I_0(q/\ft) \mathbf1_Q.
\]
 Differentiate this identity and use the QDE. We have, for every $k\ge0$,
\[
S^Q_{\tau_1H}(z)\, q^{H/z}
D^k {I^Q(q/\ft,-z)} 
=({-z})
\textstyle \sum_{j=0}^3H^jz^{-j}I_j^{[k]}(q/\ft),
\]
 where $
I_j^{[0]}=\delta_{j0}I_0$,
$I_0^{[k+1]}=DI_0^{[k]}$, and
$I_j^{[k+1]}=DI_j^{[k]}-I_{jj}I_{j-1}^{[k]}$ for $1\le j\le3$.

Quantum multiplication by $H^2$ and $H^3$ is classical. Applying QDEs along these directions, and using  the string equation, we have
\[
S^Q_{\tau_Q}(z)
=
e^{\tau_0/z}
\left(1+\frac{\tau_2H^2+\tau_3H^3}{z}\right)
S^Q_{\tau_1H}(z),
\]
where the exponential in $H^2,H^3$ truncates because $H^4=0$. Combining these identities and using $\tau_0=-120q$, we obtain
\[
\left.R(z)^{-1}\mathbf1_M\right|_Q
=
\ft^{-1/2}
\left(1+\frac{\tau_2H^2+\tau_3H^3}{z}\right)
\sum_{j=0}^3H^j
\sum_{m\ge0}
\ft^{-m}z^{m-j}
I_j^{[\mathfrak D_m]}(q/\ft).
\] 
Here for 
$\mathfrak D_m(D)=\sum_{k=0}^{2m}a_{m,k}D^k$, we define
$I_j^{[\mathfrak D_m]}
:=\sum_{k=0}^{2m}a_{m,k}I_j^{[k]}$.
All coefficients of this quintic component $R(z)^{-1}\mathbf1_M |_Q$ lie in $\mathscr R$ by Lemma \ref{formulatauq} and \ref{lem:I-ring}.

\subsubsection{The point component.}
Applying the asymptotic analysis of \cite[Section~6]{NMSP2} to the $\mathcal O(5)$-twisted theory of $\mathbb P^5$ at $p=-\mathfrak t$ gives the exponential factor $e^{5^5q/z}$. The mirror shift $J^M(0,z)=e^{-120q/z}I^M(q,z)$ then yields
\[
\tau_{\mathrm{pt}}=(5^5-120)q
\and
A^{\mathrm{pt}}=-\ft+D\tau_{\mathrm{pt}}
=-\ft+(5^5-120)q.
\]
 After removing the exponential, the Picard-Fuchs recursion and the quantum Riemann-Roch factor give, with $Y=(1-5^5q/\ft)^{-1}$,
\[
\left.R(z)^{-1}\mathbf1_M\right|_{\mathrm{pt}}
=
c\cdot Y^2\sum_{k\ge0}
f_k(Y)\left(\frac z{\ft}\right)^k\mathbf1_{\mathrm{pt}},
\qquad
f_0=1,\quad f_k\in\mathbb Q[Y]
\]
where $c^2=5/\mathfrak t^4$ by the degree-zero point pairing and the standard metric on $\Omega^{\mathrm{pt}}$. 
Hence every coefficient of $\left.R(z)^{-1}\mathbf1_M\right|_{\mathrm{pt}}$ belongs to $\mathscr R$.

\subsubsection{The remaining columns.}
The genus-zero QDEs, written using $D_p=p+zD$, are
\[
\begin{aligned}
D_pS^M(z)^* 
&=S^M(z)^*A^M,\\
D_p S^Q_{\tau_Q}(z)^* 
=S^Q_{\tau_Q}(z)^*A^Q,  &\qquad
D_p e^{\tau_{\mathrm{pt}}/z} 
=e^{\tau_{\mathrm{pt}}/z}A^{\mathrm{pt}}.
\end{aligned}
\]
Since the GRR factors are independent of $q$ and commute with $p$, differentiating the adjoint of \eqref{birk} 
gives
\[
-zDR(z)^{-1}+A^{\mathrm{loc}}R(z)^{-1}
=R(z)^{-1}A^M,\ \
\text{where} \ \ A^{\mathrm{loc}}=A^Q\oplus A^{\mathrm{pt}}.
\]
This proves the relative QDE \eqref{eq:two-QDE}. The subdiagonal form of $A^M$ then gives
\begin{equation}\label{msp-R}
R(z)^{-1}p^{j+1}=(A^{\mathrm{loc}}-zD)R(z)^{-1}p^j
-\sum_{i=0}^j(A^M)_{ij}R(z)^{-1}p^i,\qquad 0\le j<4.
\end{equation}
The unit column $R(z)^{-1}\mathbf 1_M$ and both QDE matrices have coefficients in $\mathscr R$. Differential closure and this recursion give the same conclusion for $R(z)^{-1}$. By symplecticity, the coefficients of $R(z)$ also lie in $\mathscr R$, completing the proof of Theorem~\ref{Omega01}. 
}
\medskip

The inverse CohFT reconstruction in \eqref{Omega01cohft} is also
defined over $\mathscr R$. Indeed, the scalar factors in the dilaton
resummation are $I_0$ and a nonzero constant multiple of $Y^2$, both
units of the coefficient ring; after this resummation, each stable
correlator involves only finitely many graph terms. The marked
polynomiality of $\Omega^M$ therefore gives the local CohFT over
$\mathscr R$. For $g\ge2$, the string and divisor equations and the
dimension constraint identify the integral of its quintic component over
$\Mbar_g$ with $F_g(Q(q))$. This proves $F_g(Q(q))\in\mathscr R$.

$$$$

\section{Decomposition of the $R$-matrix and the proof of Theorem \ref{thm:main}}  \label{sec:decompR}

We start with the master-space
$R$-matrix satisfying \eqref{eq:master-cohft-R} and
Definition~\ref{def:globallocal}.  
The factorization below  forces every non-leading $R^{\mathrm{prin}}$-graph contribution to  vanish. After factoring out $R^{\mathrm{reg}}$, there are only the isolated
$X$-vertex and the isolated point vertex left.


\subsection{Decomposition of the $R$-matrix}


Let $\eta^M,\eta^X,\eta^\pt$ be the corresponding pairings in ~\eqref{eq:master-cohft-R}, let $\eta^{\mathrm{loc}}=\eta^X\oplus\eta^{\mathrm{pt}}$.
We use bracketed subscripts to denote coordinates:
\[
[v]_i=\eta^{\mathrm{loc}}(v,\varphi^i)\quad(0\le i\le3),
\qquad [v]_{\mathrm{pt}}
=\eta^{\mathrm{loc}}(v,\mathbf1_{\mathrm{pt}}^\vee).
\]
For matrix entries, the first subscript denotes the row and the second denotes the column. 
Let $\sH^{\leq1}=\operatorname{span}\langle\varphi_0,\varphi_1\rangle, \sH^{\geq2}=\operatorname{span}\langle\varphi_2,\varphi_3\rangle$.

\begin{proposition}\label{prop:decompose-R}
The $R$-matrix in Definition~\ref{def:globallocal} admits a
 factorization into a product of symplectic matrices:
\[
R=R^{\mathrm{reg}}R^{\mathrm{prin}}
\]
where $R^{\mathrm{reg}}$ and its inverse are conifold regular, and
\begin{alignat}{2}
&[(R^{\mathrm{prin}})^{-1}v]_i
&&=[v]_i,\qquad i=0,1,\label{eq:first2row}\\
&[(R^{\mathrm{prin}})^{-1}]_{\mathrm{pt},\mathrm{pt}}
&&=1.\label{eq:general-point-data}
\end{alignat}
\end{proposition}

\begin{lemma}\label{lem:further-reg}
The first two rows of $R^{-1}$ are conifold regular.  For
$e_0=\mathbf1_M$ and $e_1=(A^M-zD)e_0$, the $2\times2$ matrix
\[
\begin{pmatrix}
[R^{-1}e_0]_0 &[R^{-1}e_1]_0\\
[R^{-1}e_0]_1 &[R^{-1}e_1]_1
\end{pmatrix}
\]
has a conifold-regular inverse.
\end{lemma}
\begin{proof}
We construct the basis $\{e_j\}_{j=0}^4$ by letting $e_{j+1}=(A^M-zD)e_j$, $1\le j<4$.
The matrix with columns $e_0,\ldots,e_4$ and its inverse are conifold regular,
since at $z=0$ it is the matrix in Definition~\ref{def:globallocal} (2). Set
\begin{equation}\label{eq:original-low-unit}
\mathcal R_i=[R^{-1}\mathbf1_M]_i,\qquad i=0,1,\mathrm{pt}.
\end{equation}
In the homogeneous Calabi--Yau threefold basis, the unit axiom and dimension constraint
imply that $A^X$ is lower triangular with equal diagonal entries.
Since $R(0)$ is invertible, setting $z=0$ in \eqref{eq:two-QDE} and
taking traces gives
\begin{equation}\label{eq:trace-AM-Aloc}
    [A^X]_{00}=\frac{\operatorname{tr}A^M-A^{\mathrm{pt}}}{4}.
\end{equation}
Thus $[A^X]_{00}$ is conifold regular by
Definition~\ref{def:globallocal}. The QDE then gives
\[
\binom{[R^{-1}e_j]_0}{[R^{-1}e_j]_1}
=\left[\begin{pmatrix}[A^X]_{00}&0\\{}[A^X]_{10}&[A^X]_{00}\end{pmatrix}-zD\right]^j
 \binom{\mathcal R_0}{\mathcal R_1},\qquad 0\le j\le4.
\]
The right-hand side is therefore conifold regular,
so the first two rows are conifold regular.  At $z=0$, the $2\times2$
matrix in the statement has determinant
$[A^X]_{10}\mathcal R_0(0)^2=[A^X]_{10}I_0^2$,
its inverse is a product of conifold regular
elements in $\mathscr R$.
Computing the inverse of the required $2\times2$ matrix order by order in $z$ shows
that it is conifold regular.
\end{proof}

\begin{proof}[Proof of Proposition~\ref{prop:decompose-R}]
On $H^3(X)$, we choose $R^{\mathrm{reg}}$ to be a constant isometry; its edge bivector is zero. The construction below
concerns the even part.
The factorization is obtained by \eqref{eq:first2row},
\eqref{eq:general-point-data} and the symplectic condition.

\smallskip
\noindent\textbf{Step 1.} \emph{Define $R^{\mathrm{reg}}$ on $\sH^X$.}
To eliminate $X$-insertions of degrees $0$ and $1$ at edges,
we require \eqref{eq:first2row}.
For a symplectic factorization, \eqref{eq:first2row} is equivalent to
\[
R^{\mathrm{reg}}(z)\varphi^i=R(z)\varphi^i,\qquad i=0,1.
\]
These columns are conifold regular by Lemma~\ref{lem:further-reg}, and $R^{\mathrm{reg}}$ inherits the symplecticity of $R$ on $\sH^{\geq2}$.
For $R^{\mathrm{reg}}$ on $\sH^{\leq1}$, symplecticity requires the following,
\begin{equation}\label{eq:symplectic-on-H23}
\begin{aligned}
\eta^M(R^{\mathrm{reg}}(z)\varphi_i,
       R(-z)\varphi^j)=\delta_{ij},\quad
\eta^M(R^{\mathrm{reg}}(z)\varphi_i,
       R^{\mathrm{reg}}(-z)\varphi_j)=0, \quad i,j=0,1.
\end{aligned}
\end{equation}
Let $v\in\operatorname{span}\langle e_0,e_1\rangle$, by Lemma \ref{lem:further-reg}, the  map
\[
v
\longmapsto
\binom{\eta^M(v,R(-z)\varphi^0)}
      {\eta^M(v,R(-z)\varphi^1)}
=\binom{[R(z)^{-1}v]_0}{[R(z)^{-1}v]_1}
\]
has a conifold-regular inverse,
so the first equation in (\ref{eq:symplectic-on-H23}) uniquely determines
$R^{\mathrm{reg}}(z)\varphi_0$ and $R^{\mathrm{reg}}(z)\varphi_1$
as the preimages of $\binom10$ and $\binom01$ respectively.
To satisfy the second equation in (\ref{eq:symplectic-on-H23}), we further make the following adjustment for $i=0,1$,
\[
R^{\mathrm{reg}}(z)\varphi_i
\longmapsto
R^{\mathrm{reg}}(z)\varphi_i
-\frac12\sum_{j=0}^1
\eta^M(R^{\mathrm{reg}}(z)\varphi_i,
       R^{\mathrm{reg}}(-z)\varphi_j)R(z)\varphi^j.
\]
Each adjustment term subtracts half of the original pairing in the
second equation. The two adjustment terms pair to zero, so the first
equation is unchanged and the second pairing becomes zero. 
And each of $\{R^{\mathrm{reg}}(z)\varphi_i\}_{i=0}^3$ remains conifold regular because
the adjustment terms are conifold regular.

\smallskip
\noindent\textbf{Step 2.} \emph{Define $R^{\mathrm{reg}}\mathbf1_{\mathrm{pt}}$.}
If  $R^{\mathrm{reg}}(z)\mathbf1_{\mathrm{pt}}
= R(z)\mathbf1_{\mathrm{pt}}$, then   (\ref{eq:general-point-data}) is satisfied. To satisfy symplecticity, we make the adjustment
\begin{equation}\label{eq:point-column-completion}
R^{\mathrm{reg}}(z)\mathbf1_{\mathrm{pt}}
= R(z)\mathbf1_{\mathrm{pt}}
-\sum_{j=0}^1
\eta^M\!\left(R(z)\mathbf1_{\mathrm{pt}},
             R^{\mathrm{reg}}(-z)\varphi_j\right)
R(z)\varphi^j,
\end{equation}
which preserves \eqref{eq:general-point-data} because the correction lies in $\sH^{\geq2}$ after applying $R(z)^{-1}$.
\eqref{eq:symplectic-on-H23} makes this column orthogonal to $R^{\mathrm{reg}}(-z)\sH^{\leq1}$.
Because the correction lies in the   subspace $R(z)\sH^{\geq2}$ and is orthogonal to $R(-z)\mathbf1_{\mathrm{pt}}$, it preserves orthogonality to $R(-z)\sH^{\geq2}$ and the pairing $\eta^{\mathrm{pt}}$. Hence $R^{\mathrm{reg}}$ is symplectic, and $R^{\mathrm{prin}}=(R^{\mathrm{reg}})^{-1}R$  is symplectic.

\smallskip
\noindent\textbf{Step 3.} \emph{Conifold regularity of $R^{\mathrm{reg}}(z)\mathbf1_{\mathrm{pt}}$.}
By (\ref{eq:first2row})  and (\ref{eq:point-column-completion}), in the basis
$\{\varphi_0,\ldots,\varphi_3,\mathbf1_{\mathrm{pt}}\}$, we have
\begin{equation}\label{eq:residual-matrix-display}
R^{\mathrm{prin}}(z)^{-1}=R(z)^{-1}R^{\mathrm{reg}}(z)=
\begin{pmatrix}
I_{2\times2}&0_{2\times2}&0_{2\times1}\\
*_{2\times2}&I_{2\times2}&*_{2\times1}\\
*_{1\times2}&0_{1\times2}&1
\end{pmatrix}.
\end{equation}
Taking determinant gives
$\det R^{\mathrm{reg}}(z)=\det R(z)$.
We have
\begin{equation}\label{eq:diff-det-R}
zD\bigl(\det R(z)^{-1}\bigr)
=z\det R(z)^{-1}\,\operatorname{tr}[R(z)D(R(z)^{-1})]
=\bigl(\operatorname{tr}A^{\mathrm{loc}}-\operatorname{tr}A^M\bigr)
\det R(z)^{-1}=0.
\end{equation}
Here, the first equality uses the determinant differentiation formula; the second equality uses QDE \eqref{eq:two-QDE}; setting $z=0$ in  \eqref{eq:two-QDE} gives $\operatorname{tr}A^{\mathrm{loc}}=\operatorname{tr}A^M$ and the third equality.
Thus  every $z$-coefficient of
$\det R(z)$ is independent of $\Delta$.
Hence $\det R^{\mathrm{reg}}(z)$ is conifold regular.

By Step~1 and the symplecticity, the first four
rows of $R^{\mathrm{reg}}(z)^{-1}$ are conifold regular.
For each $j$, replace the last row of $R^{\mathrm{reg}}(z)^{-1}$ by row $j$ of
$I_{5\times5}$. The resulting matrix is conifold regular. Multiplying it on the right
by $R^{\mathrm{reg}}(z)$ gives
\[
\begin{pmatrix}
I_{4\times4}&0_{4\times1}\\
*_{1\times4}&[R^{\mathrm{reg}}(z)]_{j,\mathrm{pt}}
\end{pmatrix}.
\]
Taking determinant shows that $[R^{\mathrm{reg}}(z)]_{j,\mathrm{pt}}$ is regular, because $\det R^{\mathrm{reg}}(z)$ is regular and the replaced matrix has regular entries.
This holds for every $j$, so $R^{\mathrm{reg}}(z)\mathbf1_{\mathrm{pt}}$
is regular.
Symplecticity then gives regularity of $R^{\mathrm{reg}}(z)^{-1}$.
\end{proof}

\begin{corollary}\label{prop:V-prin}
The edge bivector $V_{R^{\mathrm{prin}}}(z,w)$ has no term containing
$\varphi_0$ or $\varphi_1$ in either tensor factor, and its
$\mathbf1_{\mathrm{pt}}\otimes\mathbf1_{\mathrm{pt}}$ coefficient is zero.
\end{corollary}
\begin{proof}
We represent a bivector by a matrix with $(i,j)$-entry the coefficient of
$\varphi_i\otimes\varphi_j$ and indices ordered $0,1,2,3,\mathrm{pt}$.
By \eqref{eq:residual-matrix-display}, the edge bivector  is of the form
\begin{equation*}
V_{R^{\mathrm{prin}}}(z,w)
=\begin{pmatrix}
0_{2\times2}&0_{2\times2}&0_{2\times1}\\
0_{2\times2}&*_{2\times2}&*_{2\times1}\\
0_{1\times2}&*_{1\times2}&0
\end{pmatrix}.
\end{equation*} 
Here we have used {\small $V_{R^{\mathrm{prin}}}(z,w)
=\frac{1}{z+w}\sum_\alpha
\Bigl(\varphi_\alpha\otimes\varphi^\alpha
-R^{\mathrm{prin}}(z)^{-1}\varphi_\alpha
\otimes R^{\mathrm{prin}}(w)^{-1}\varphi^\alpha\Bigr)$}.
\end{proof}

\subsection{Proof of Theorem~\ref{thm:main}}



\begin{lemma}\label{lem:CY3-vertex-elimination}
Let \(\Omega\) be the Gromov-Witten CohFT of a Calabi--Yau threefold $X$ on
\(H^\bullet(X)\), \(\tau\in H^{\mathrm{even}}(X)\) and \(T(z)\in zH^{\mathrm{even}}(X)[[z]]\).  
Let \(T \Omega^\tau\) denote
the CohFT obtained by first applying the primary shift by $\tau$ and then the
\(T\)-action.  If \(2g-2+n>0\), \(n\geq1\), and
\(\gamma_j\in H^{2d_j}(X)\) are homogeneous with \(d_j>1\) (allowing half-integral degrees), then
\[
  (T \Omega^\tau)_{g,n}
   (\gamma_1,\ldots,\gamma_n)=0
 \qquad\text{in}\qquad
 H^\bullet(\overline{\mathcal M}_{g,n}).
\]
\end{lemma}

\begin{proof}
The unit axiom removes the unit part of the primary shift.
Since $T(z)=O(z)$, a term with $m$ extra markings has insertion degree
at least $\sum_jd_j+m>n+m$, its virtual dimension, and hence vanishes.
\end{proof}

\begin{proof}[Proof of Theorem~\ref{thm:main}]
Apply $(R^{\mathrm{reg}})^{-1}$ from Proposition~\ref{prop:decompose-R} to \eqref{eq:master-decomp} and consider the graph sum for $\big((R^{\mathrm{reg}})^{-1}\Omega^M\big)_{(g,0)}$. By Corollary~\ref{prop:V-prin}, there are no point--point edges, so every graph with an edge contains an $X$-vertex incident to an edge. All edge insertions at this vertex have degree greater than one: at least two in the even sector and $3/2$ in the odd sector. The contribution of such a graph therefore vanishes by Lemma~\ref{lem:CY3-vertex-elimination}. Only the isolated $X$-vertex and the isolated point vertex remain, giving
\begin{equation}\label{eq:final}
\int_{\Mbar_g}R^{\mathrm{reg}}(z)^{-1}\Omega^M
=\int_{\Mbar_g}R^{\mathrm{prin}}
\bigl(T_R|_X\Omega^{X,\tau}\oplus
      T_R|_{\mathrm{pt}}\Omega^{\mathrm{pt}}\bigr)
=I_0^{2-2g}F_g^X
+\int_{\Mbar_g}T_R|_{\mathrm{pt}}\Omega^{\mathrm{pt}}.
\end{equation}
The factor $I_0^{2-2g}$ comes from the dilaton equation and
Condition (3) of Definition \ref{def:globallocal}.

The left-hand side of \eqref{eq:final} is regular, since it is a finite
sum of products of regular edge and leg contributions and regular correlators of $\Omega^M$.
Gap property in  condition (iii) gives
\[
\iota\int_{\Mbar_g}T_R|_{\mathrm{pt}}\Omega^{\mathrm{pt}}
=\kappa^{1-g}\frac{(-1)^gB_{2g}}{2g(2g-2)}
\omega_1^{2-2g}+O(\Delta^0).
\]
Substituting into \eqref{eq:final} and using
$\tau_c=\omega_1/\omega_0$, we 
obtain
$F_g^{X,c}
=\kappa^{1-g}\frac{(-1)^{g-1}B_{2g}}{2g(2g-2)}
\tau_c^{2-2g}+O(\tau_c^0)$
and finish the proof.
\end{proof} 


\section{The conifold gap for cone-vertex  theories  } \label{sec:TKW}


For each formal coordinate $\omega(\Delta)=\Delta+O(\Delta^2)$, we define a translation $T_\omega$ of the point CohFT by a formal Laplace transform. In quintic MSP theory, the corresponding translation encodes the contributions of rational tails attached to the vertex $O$ of the cone $C_Q$. We call $T_\omega$ the \emph{cone-Laplace translation} and $T_\omega\Omega^{\mathrm{pt}}$ the associated \emph{cone-vertex theory}.

Theorem~\ref{thm:rankone} proves the universal cone-vertex gap: $T_\omega\Omega^{\mathrm{pt}}$ has the  conifold gap property with respect to $\omega$, with $\kappa=1$. The proof uses localization on the theory of an MSP-like master space extension of the resolved conifold. In this setting, the localization formula directly relates the pole parts of the level-$0$ and level-$1$ contributions.

%



\subsection{An extension of the resolved conifold} For the resolved conifold $\operatorname{Tot}_{\PP^1}(\mathcal O(-1)^{\oplus2})
$, we consider  the master space $Z=\operatorname{Tot}_{\PP^2}(\mathcal O(-1)^{\oplus2})$. 
Let $G=\CC^*$ act on $\PP^2=\{[x_1:x_2:\mu]\}$ by scaling $\mu$, and let $\ft$ be the equivariant parameter. We set $\ft=1$.

The fixed locus is $\PP^1=\{\mu=0\}\subset\PP^2$ and $\pt=[0:0:1]$.
{Choose a lift of the action so that the fibers of each $\mathcal{O}(-1)$ have weight $0$ along $\PP^1$ and $\ft$ at $\pt$.}
Let $Z^{[0]}$ be the local $\PP^1$ fixed locus,  decorated by level 0, let $\pt$ be decorated by level 1,  one has
\begin{equation}\label{eq:rc-fixed}
(\PP^2)^G=\PP^1\sqcup\pt, \quad
 Z^{G}=Z^{[0]}\sqcup\pt, \quad
 N_{Z^{[0]}/Z}=\mathcal O_{\PP^1}(1)\otimes\CC_{\ft},\quad
 T_{\pt}Z=\CC_{-\ft}^{\oplus2}\oplus\CC_{\ft}^{\oplus2}.
\end{equation}
By abuse of notation, as in Section~\ref{sec:msp}, 
let
$p=c_1^{G}(\mathcal O_{\PP^2}(1))$ and $H=c_1(\mathcal O_{\PP^1}(1))$; then  $p|_{\PP^1}=H$, $p|_{\pt}=-\ft$. 
The cohomology ring $\CC(\ft)[p]/\bigl(p^2(p+\ft)\bigr)$ has basis $\{ 1,p,p^2\}$ and the fixed locus basis
$\{\mathbf1_{\PP^1},H,\mathbf1_{\pt}\}$, with relation
\[1=\mathbf1_{\PP^1}+\mathbf1_{\pt},\qquad p=H-\ft\mathbf1_{\pt},\qquad p^2=\ft^2\mathbf1_{\pt}.\]
The pairing is
$(x,y)^Z
 =(x|_{\PP^1},y|_{\PP^1})^{\PP^1,\mathrm{tw}}
  +(x|_{\pt},y|_{\pt})^{\pt,\mathrm{tw}},$
with $(\mathbf1_{\pt},\mathbf1_{\pt})^{\pt,\mathrm{tw}}=1$.

Let $\pi:\mathcal{C}\to\Mbar_{g,n}(\PP^2,d)$,
$f:\mathcal{C}\to\PP^2$ be the universal curve and universal map, and set
$\mathcal{L}=f^*\mathcal O(1)$. 
When $d>0$, $R^1\pi_*\mathcal{L}^{-1}$ is a vector bundle. We regard $\Mbar_{g,n}(Z,d)$ as the analogy of the MSP moduli for the resolved conifold. Its  virtual cycle is
\[
 \begin{aligned}
 [\Mbar_{g,n}(Z,d)]^{\vir}
 =e_G\!\left((R^1\pi_*\mathcal{L}^{-1})^{\oplus2}\right)
 \cap[\Mbar_{g,n}(\PP^2,d)]^{\vir}\in A^G_{d+1-g+n}\bigl(\Mbar_{g,n}(\PP^2,d)\bigr).
 \end{aligned}
\]

\medskip

For $2g-2+n>0$, let
$\pr:\Mbar_{g,n}(Z,d)\longrightarrow\Mbar_{g,n}$ be the
forgetful map. Let $\bar\psi_i$ be the
pullback of the $i$-th cotangent line class on $\Mbar_{g,n}$.
Put $\Delta=1-\sfq$. We consider insertions
\begin{equation*}
\gamma(\Delta,z)=\gamma^{[1]}p^2-\gamma^{[0]}p(p+1),
\qquad \gamma^{[0]},\gamma^{[1]}\in\CC[[\Delta,z]].
\end{equation*}
Their restrictions are  \vspace{-0.2cm}
\begin{equation}\label{eq:rc-lift}
\left.\gamma\right|_{\PP^1}=-\gamma^{[0]}H,
\qquad \left.\gamma\right|_{\pt}=\gamma^{[1]}.
\end{equation}
The dimension argument used in Section~\ref{sec:msp} gives the following.
\begin{lemma}\label{lem:rc-global-regularity}
For insertions $\gamma_1,\ldots,\gamma_n$ of the above form, the
positive-degree ancestor series is regular at $\Delta=0$:
\begin{equation}\label{eq:rc-global}
\textstyle \sum_{d>0}\sfq^d\pr_*\!\Big(
\textstyle \prod_{i=1}^n(\ev_i^*\gamma_i)
 \cap[\Mbar_{g,n}(Z,d)]^{\vir}
 \Big)\in A^*(\Mbar_{g,n})_{\CC}[[\Delta]].
\end{equation}
\end{lemma}
\begin{proof}
Before inverting and specializing $\ft$, each $\sfq^d$ coefficient has pushforward degree
\[
 \begin{aligned}
 \deg\pr_*(\cdots)
 &=(3g-3+n)-(d+1-g+n)+\text{insertion degree}
 \\
 &=4g-4-d+\text{insertion degree}
 \end{aligned}
\]
in $A^*(\Mbar_{g,n})[\ft][[\Delta]]$.
Ancestor degrees are bounded by $\dim\Mbar_{g,n}$, so the pushforward
vanishes for $d\gg0$. The sum in \eqref{eq:rc-global} is therefore finite
and conifold regular.
\end{proof}

We study fixed loci that contributes to \eqref{eq:rc-global}. First the fixed loci must have positive degree.
\begin{lemma}\label{lem:rc-graphs}
Only the following two types of fixed loci contribute to \eqref{eq:rc-global}:
\begin{enumerate}
\item[(i)] maps with image equal to $\PP^1$;
\item[(ii)] {maps whose localization graph has a single stable vertex of genus $g$ over $\pt$, with rational edges attached to it, each with at most one marking.}
\end{enumerate}
\end{lemma}
\begin{proof}

 Suppose the fixed loci contains an edge (connecting level $0$ to level $1$) $e$ incident to a label $0$ vertex $v$. Supposse $v$ has positive degree.
 Take arbitrary $f:C\to \PP^2$ be in the fixed loci where $C_e\cup_x C_v\sub C$ and $C_v$ is a connected curve mapped to $\PP^1$. Then $\Gamma(C_e,f^*\mathcal O(-1))=0=\Gamma(C_v,f^*\mathcal O(-1))$ holds in families. 
 By normalization sequence  the family version of $f\sta\mathcal{O}(-1)|_x$ forms a  subbundle of $R^1\pi_*\mathcal{L}^{-1}$. And this subbundle has equivariant euler class to be $ev_x\sta(-H)$. Note here we use that $\mathcal{O}(-1)$ has weights $0$ along $\PP^1$.
 Now that $H^2=0$ on $\PP^1$, the fixed loci 
has vanishing virtual cycle. The same conclusion holds if $C$ contains a subcurve $C_e\cup_x C_{e'}$ with two edges $e,e'$.

 Now assume that the fixed loci contains a level zero vertex $v$ of degree zero and genus $g_v$. Suppose $v$ has $k_v$  incident edges and $l_v$ markings. 
 Take arbitrary $f:C\to \PP^2$ in the fixed loci where $C_v\sub C$ corresponds to $v$ and  contracted to a point $f(C_v)=q\in\PP^1$. Then $H^0(C_v,f^*\mathcal{O}(-1)|_{C_v})\simeq \mathcal{O}(-1)|_{q}$.
 Normalize the $k_v$ many node for the sheaf $\sL^{-1}$ associates  an exact sequence 
 $$ 0\to \mathcal{O}(-1)|_{q}\to \mathcal{O}(-1)|_{q}^{\oplus k_v} \to H^1(C,f\sta \mathcal{O}(-1)).$$

There are three cases:
\begin{enumerate}
\item $k_v\geq2$, the remaining node contribution gives $0$ by same reason.
\item $k_v=1,g_v=0$.  Stability requires
$l_v\geq 2$. By~\eqref{eq:rc-lift} every such marking receives $H$, and the contribution to \eqref{eq:rc-global} vanishes because $H^2=0$.
\item $k_v=1,g_v>0$.
Over the  moduli space $\barM_{g_v,1+l_v}\times \PP^1$ associated with $v$, the obstruction theory and the $\mathcal O(-1)$ twist give the factor 
\begin{equation}\label{eq:rc-line-vertex}
e\!\left(\mathbb E^\vee\boxtimes\mathcal O_{\PP^1}(-1)\right)^2
 e\!\left(\mathbb E^\vee\boxtimes T_{\PP^1}\right)
=(-1)^{3g_v}\bigl(\lambda_{g_v}+H\lambda_{{g_v}-1}\bigr)^2
 \bigl(\lambda_{g_v}-2H\lambda_{{g_v}-1}\bigr)=0.
\end{equation}
The vanishing uses Mumford's relation $\lambda_{g_v}^2=0$.
\end{enumerate}

Thus only fixed loci of types (i) and (ii) can contribute.
\end{proof}

In case~(ii) of Lemma~\ref{lem:rc-graphs}, an edge without marking is
called a \emph{rational tail}, and an edge with a marking a
\emph{rational end}; the marking of an end lies over $\PP^1$. Contract
tails and ends to the stable vertex, sum over their degrees with weight
$\sfq^d$, and pass to ancestor classes. The tails then give a translation
$T_{-\log\sfq}(z)$ at the stable vertex, and an end whose marking carries
$-H$ gives a leg factor $E(\Delta,\bar\psi)$. 
 
\begin{lemma}\label{lem:rc-contributions}
Let $D=\sfq\partial_{\sfq}=-\sfq\partial_\Delta$. Then 
$E,T_{-\log\sfq}$ are the unique elements
of $\sfq\QQ[[\sfq,z]]$ satisfying  \vspace{-0.2cm}
\begin{align}\label{eq:rc-E-Ttau}
 (\Delta+zD)\bigl(1+E\bigr)=1,\qquad
 (\Delta+zD)\bigl(z-T_{-\log\sfq}\bigr)=z(1+E).
\end{align}
Moreover, the solutions $E$ and $T_{-\log\sfq}$ lie in $\QQ[\Delta^{-1}][[z]]$.
\end{lemma}

\begin{proof}
Let $J^{\pt}$ be the restriction to $\pt$ of the small $J$-function of $Z$.
As in \cite[Lemma~1.8 and Section~3]{NMSP2}, in descendants the tails give
the series $L(z)=[z+J^{\pt}(-z)]_+$, and an end gives
$[-z^{-1}(1+zD)(J^{\pt}(-z)+z)]_+$ by the divisor equation (writing
$-H=-p-p^2$, the $p^2$-term only gives negative powers of $z$). By the toric
mirror theorem,
\[
 J^{\pt}(-z)=-z\sum_{d\ge0}\frac{(-\sfq/z)^d}{d!\,(1+dz)^2},
 \qquad\text{so}\qquad
 (1+zD)^2J^{\pt}(-z)=-ze^{-\sfq/z}.
\]
As $e^{\sfq/z}(1+zD)e^{-\sfq/z}=\Delta+zD$, the series
$F:=e^{\sfq/z}J^{\pt}(-z)\in \mathbb Q((z))[[\sfq]]$ satisfies $F|_{q=0}=-z$ and $(\Delta+zD)^2F=-z$. 
Writing $F=\sum_{d\ge0}F_d(z)\sfq^d$, the equation gives
\[
(1+dz)^2F_d
=\bigl(2+(2d-1)z\bigr)F_{d-1}-F_{d-2},
\qquad F_{-1}=0,\quad F_0=-z.
\]
Induction on $d$ gives
$F\in z\QQ[[\sfq,z]]$.
On the other hand,  we rewrite $(\Delta+zD)^2F=-z$ as
\[
F=-z\Delta^{-2}
+z(\Delta^{-2}-\Delta^{-1})F
-2z\Delta^{-1}DF-z^2\Delta^{-2}D^2F.
\]  
Since $D$ preserves $\QQ[\Delta^{-1}]$, induction on the degree of $z$ shows $F\in z\,\QQ[\Delta^{-1}][[z]]$.

By the string and dilaton equations, passing to ancestors replaces $L$ by
$[e^{u/z}(L-u)]_+$ and a descendant leg $f$ by $[e^{u/z}f]_+$, where
$u\in\sfq\QQ[[\sfq]]$ solves $u=[e^{u/z}L(z)]_{z^0}$
\cite[Appendix~2]{CG}. Since $F$ has only positive powers of $z$ and
$e^{\sfq/z}[J^{\pt}(-z)]_-$ only negative ones, $u=\sfq$ satisfies the above condition, and
\[
 T_{-\log\sfq}=\bigl[e^{\sfq/z}(L(z)-\sfq)\bigr]_+=z+F,\qquad
 1+E=-z^{-1}(\Delta+zD)F .
\]
Thus $T_{-\log\sfq}\in \QQ[\Delta^{-1}][[z]]$. And both identities in \eqref{eq:rc-E-Ttau} follow from
$(\Delta+zD)^2F=-z$, they uniquely  determine $E$ and $T_{-\log\sfq}$ from the vanishing at $\sfq=0$.
Finally, $(\Delta+zD)\bigl(1+E\bigr)=1$ gives $E=\Delta^{-1}-1-z\Delta^{-1}DE$, which gives $E\in \QQ[\Delta^{-1}][[z]]$ by induction.
\end{proof}

We use the master space $Z$ to transfer the pole part of the resolved conifold theory to the translated point theory.

Write $[\pt]_{g,n}\in A_0(\Mbar_{g,n})$ for the unique degree-one tautological class (\cite[Theorem~1.1]{GV01}).
For $d>0$, set
\begin{equation}\label{eq:rc-local-cycle}
 C_{g,n,d}:=\pr_*\!\left(
 e\!\left((R^1\pi_*f^*\mathcal O_{\PP^1}(-1))^{\oplus2}\right)
 \prod_i\ev_i^*H\cap[\Mbar_{g,n}(\PP^1,d)]^{\vir}
 \right)\in A_0(\Mbar_{g,n}).
\end{equation}
By virtual localization, $C_{g,n,d}$ is tautological. By divisor equation, $(\pi_1)_*C_{g,n+1,d}=dC_{g,n,d}$, where $\pi_1$ forgets the last marking. Then multiple-cover formula (see \cite{FP})
 gives
\begin{equation}\label{eq:rc-local-value}
 C_{g,n,d}=
 \frac{(-1)^{g+1}(2g-1)B_{2g}}{(2g)!}\,d^{2g-3+n}[\pt]_{g,n}.
\end{equation}

\begin{theorem}\label{thm:rc-class-gap}
For $i=1,...,n$, let 
$\gamma^{[1]},\gamma^{[0]},\gamma_i^{[1]},\gamma_i^{[0]}
\in\CC[[\Delta,z]]$ satisfy
$\gamma^{[1]}+\gamma^{[0]}E=O(z^2)$. 
Then for $2g-2+n>0$, we have the following equality in tautological Chow ring,
\begin{equation}\label{eq:rc-class-gap}
 \begin{split}
 &\bigl[(T_{-\log\sfq}+\gamma^{[1]}+\gamma^{[0]}E)
   \Omega^{\pt}\bigr]_{g,n}\bigl(
   \gamma_1^{[1]}+\gamma_1^{[0]}E,\ldots,
   \gamma_n^{[1]}+\gamma_n^{[0]}E\bigr)\\
 =&
 \frac{(-1)^{g+n}(2g-1)(2g-3+n)!B_{2g}}
 {(2g)!\,\bigl(-\log\sfq+\gamma^{[0]}(\Delta,0)\bigr)^{2g-2+n}}
 [\pt]_{g,n}\prod_{i=1}^n\gamma_i^{[0]}(\Delta,0)+O(\Delta^0).
 \end{split}
\end{equation}
\end{theorem}
\begin{proof}

Localize \eqref{eq:rc-global} with insertions $\gamma_1,\ldots,\gamma_n$ and $m$ copies of $\gamma$, then forget the last $m$ markings and sum with weight $1/m!$:
\[
 \Mbar_{g,n+m}(Z,d)\xrightarrow{\pr}\Mbar_{g,n+m}
 \xrightarrow{\pi_m}\Mbar_{g,n},
\]
where $\pr$  forgets the map and $\pi_m$ forgets the last $m$ markings.

By Lemma~\ref{lem:rc-graphs}, there are two types of contributions.

At $\pt$, Mumford's relation cancels the twisted Hodge factor.
Lemma~\ref{lem:rc-contributions} gives the translation
$T_{-\log\sfq}+\gamma^{[1]}+\gamma^{[0]}E$ and  legs
$\gamma_i^{[1]}+\gamma_i^{[0]}E$, hence the left-hand side of
\eqref{eq:rc-class-gap}.

After dilaton resummation of the linear term of $T_{-\log\sfq}$, only terms with $m\leq3g-3+n$ contribute, since $\gamma^{[1]}+\gamma^{[0]}E=O(z^2)$.

At $Z^{[0]}$, by the dimension constraint, only $\gamma_i^{[0]}(\Delta,0)$ and $\gamma^{[0]}(\Delta,0)$ contribute,  giving
\begin{equation}\label{eq:rc-localization}
(-1)^{n}\prod_i\gamma_i^{[0]}(\Delta,0)
       \sum_{m=0}^{3g-3+n}\frac{(-\gamma^{[0]}(\Delta,0))^m}{m!}
       \sum_{d>0}\sfq^d (\pi_m)_*C_{g,n+m,d}.
\end{equation}
The summands with $m\geq2g-2+n$ are conifold regular.
In fact, since $\gamma$ and $\gamma_i$ have coefficients in $\CC[[\Delta,z]]$, any poles come from the $\sfq$-series.
Using (\ref{eq:rc-local-value}), expand the potential in conifold coordinate,
\[
\begin{aligned}
 \sum_{d>0}\sfq^d d^m C_{g,n,d}
 &=\frac{(-1)^{g+1}(2g-1)B_{2g}(2g-3+n+m)!}
 {(2g)!(-\log\sfq)^{2g-2+n+m}}[\pt]_{g,n}+O(\Delta^0),
 \end{aligned}
\]
which is the gap condition for local $\PP^1$.
Since $-\log\sfq=\Delta+O(\Delta^2)$, its
pole order is at most $2g-2+n+m$.
By $E=\sfq/\Delta-z\sfq/\Delta^3+O(z^2)$ and $\gamma^{[1]}+\gamma^{[0]}E=O(z^2)$, we get $\gamma^{[0]}(\Delta,0)\in\Delta^2\CC[[\Delta]]$,
so the factor
$\frac{(-\gamma^{[0]}(\Delta,0))^m}{m!}=O(\Delta^{2m})$.
Hence the corresponding summand in \eqref{eq:rc-localization} has
$\Delta$-order at least $2m-(2g-2+n+m)=m-(2g-2+n)$,
so all terms with $m\ge2g-2+n$ have no negative powers of $\Delta$.

Thus, summing over $m\geq0$ in (\ref{eq:rc-localization}) leaves the pole part unchanged.
Then \eqref{eq:rc-class-gap} follows by  expansion in $-\log\sfq$ for fixed $m$ and then using the binomial formula:   
\begin{equation*}
\sum_{m\ge0}\frac{(-\gamma^{[0]}(\Delta,0))^m}{m!}
 \sum_{d>0}\sfq^d d^{2g-3+n+m}
=
 \frac{(2g-3+n)!}
 {(-\log\sfq+\gamma^{[0]}(\Delta,0))^{2g-2+n}}
 +O(\Delta^0).
\end{equation*} \vspace{-0.2cm}
\end{proof}


\begin{remark}
Let $n=0,\gamma=0$ in (\ref{eq:rc-class-gap}), the lowest order in $\Delta$ recovers Kazarian-Norbury $K$-class relations in Chow groups, see \cite{KN}\cite{CGG}. And see \cite{GJ26} for another proof.
\end{remark}

\subsection{The cone-Laplace transform}
\label{sec:laplace}
The Picard--Fuchs equation for the mirror of $C_Q$ can be viewed
as a cone extension of that for the mirror quintic.
This viewpoint motivates the formal construction below.

We consider the hypergeometric equation of Calabi--Yau type:
\begin{equation}\label{eq:general-PF}
 \bigg(\prod_{i=1}^n(D+\alpha_i)
       -\sfq\prod_{i=1}^n(D+\beta_i)\bigg)\omega=0,
 \qquad n\geq3,
\end{equation}
with 
$D=\sfq\partial_{\sfq}=-(1-\Delta)\partial_\Delta$, such that 
  \begin{equation}\label{eq:conifold-condition}
  \textstyle n+\sum_i\alpha_i-\sum_i\beta_i=2.
  \end{equation}

The indicial roots at $\Delta=0$ are
$0,1,\ldots,n-2,n-1+\sum_i\alpha_i-\sum_i\beta_i$.
 Under condition~\eqref{eq:conifold-condition}, we find $1$ is the double root.
There is a unique solution $\omega_1(\Delta)$  \footnote{Equivalently, $\omega_{n-1}=\frac{1}{2\pi i}\omega_{1}\log \Delta+\omega^{\prime}_{n-1}$
and $\omega^{\prime}_{n-1}$ is holomorphic at $\Delta=0$.
Actually $\omega_1(\Delta)$ is just the solution obtained
by the standard Frobenius method  for the double root 
1 without logarithms.}
such that
\begin{enumerate}
    \item $\omega_1(\Delta)=\Delta+O(\Delta^2)$;
    \item $\omega_1(\Delta)$ can be extended to a basis of solution space
    $\{\omega_0,\omega_1,...,\omega_{n-1}\}$ such that the monodromy matrix $\mathcal T$
    is given by $\mathcal T(\omega_j)=\omega_j$ for $1\leq j\leq n-2$
    and $\mathcal T(\omega_{n-1})=\omega_{n-1}+\omega_1$.
\end{enumerate}

Inspired by \cite[section 6]{NMSP2}, we study asymptotic solutions to the following {cone extended} Picard--Fuchs  equation
\begin{equation}\label{eq:general-extended}
  \bigg(\prod_{i=1}^n(D_p+\alpha_i z)(D_p+1)
        -\sfq\prod_{i=1}^n(D_p+\beta_i z)\bigg)\mathcal I=0,
  \qquad D_p=zD-1.
\end{equation}

\begin{propositiondefinition}\label{defn:conifoldlaplace}
Let
\( \Phi(x)=x-\log(1+x)\).
For every \(\omega(\Delta)\in\Delta\mathbb C[[\Delta]]\), set
\[
  A_\omega(x):=\frac{1}{1+x} \cdot \omega\!\left(\frac{x}{1+x}\right)
\]
and, for $k\geq0$, define
\begin{equation}\label{eq:Rcoeff}
  \cR_k[\omega](\Delta)
  :=
  \sum_{r=0}^k
  \frac{(k+r+1)!}{r!}\,
  \Delta^{-(k+r+2)}
  [x^{k+r+1}]
  \bigl(A_\omega(x)\Phi(x)^r\bigr).
\end{equation}
The  \emph{cone-Laplace transform} of $\omega$ and the associated  \emph{cone-Laplace translation} are defined by
\[
\mathcal R_\omega(\Delta,z)
:=
\sum_{k\geq0}(-1)^k\mathcal R_k[\omega](\Delta)z^k,
\qquad
T_\omega(z):=z\bigl(1-\mathcal R_\omega(\Delta,z)\bigr).
\]
They have the following properties.\begin{itemize}
    \item If $\omega(\Delta)=\Delta+O(\Delta^2)$, then
\begin{equation}\label{eq:extended-normalization}
  \cR_0[\omega]=\Delta^{-2},
  \qquad
  \Delta^2\cR_k[\omega]\longrightarrow0
  \quad(\Delta\to\infty,\ k>0).
\end{equation}
\item  Suppose that   \eqref{eq:conifold-condition} is satisfied, and let $\omega_1$ be the normalized vanishing period of \eqref{eq:general-PF} as above. Then
\begin{equation}\label{eq:normalized-extended-solution}
  \mathcal I_{\omega_1}
  :=
  ze^{\sfq/z}
  \sum_{k\geq0}\cR_k[\omega_1](\Delta)z^k
\end{equation}
solves \eqref{eq:general-extended} and  is the unique formal
solution of this form satisfying
\eqref{eq:extended-normalization}.
\end{itemize}
 
\end{propositiondefinition}

\begin{proof}
The normalization \eqref{eq:extended-normalization} follows immediately
from \eqref{eq:Rcoeff}: if
\(\omega(\Delta)=\Delta+O(\Delta^2)\), then
\(A_\omega(x)=x+O(x^2)\), while for \(k>0\) every summand of
\(\cR_k[\omega]\) contains a power \(\Delta^{-k-r-2}\).

We now specialize to the normalized vanishing period \(\omega_1\).
Regard the local solutions of \eqref{eq:general-PF} as functions of
\(x\) through the substitution
\[
  \Delta=\frac{x}{1+x},
\]
and let \(\omega_{\log}\) be a logarithmic solution whose monodromy
jump is \(\omega_1\).  Consider
\[
  \mathcal H[u]
  :=
  \frac1z\int_\Gamma
  \exp\left(
    \frac{\sfq(1+x)-\log(1+x)}{z}
  \right)
  u\,\frac{dx}{1+x}.
\]
Here \(\Gamma\) is a keyhole contour around the cut
\([0,\infty)\), oriented so that
\[
  \omega_{\log,+}-\omega_{\log,-}=\omega_1.
\]

Integration by parts gives, for every scalar $\eta$,
\begin{equation}\label{eq:general-intertwining}
  (D_p+\eta z)\mathcal H[u]
  =-z\mathcal H[((1+x)\partial_x-\eta)u],
  \qquad
  (D_p+1)\mathcal H[u]
  =\sfq\,\mathcal H[(1+x)u].
\end{equation}
Hence \(\mathcal H[\omega_{\log}]\) satisfies
\eqref{eq:general-extended}.

Since $\omega_{\log}=O(1)+O(x\log x)$ near $x=0$, its small-circle contribution tends to zero. Collapsing $\Gamma$ onto the cut leaves only the jump $\omega_1$. Using
\[
  \sfq(1+x)-\log(1+x)
  =\sfq-\Delta x+\Phi(x),
\]
{Laplace's method gives the asymptotic expansion}
\[
  \mathcal H[\omega_{\log}]
  \sim
  ze^{\sfq/z}
  \sum_{k\geq0}\cR_k[\omega_1](\Delta)z^k
  =
  ze^{\sfq/z}\cR_{\omega_1}(\Delta,-z),
  \qquad z\to0,
\]
where the coefficients are given by \eqref{eq:Rcoeff}.
{The resulting coefficient identities are rational in $\Delta$, so they also hold as Laurent series at $\Delta=0$.}

For uniqueness, conjugating \eqref{eq:general-extended} by
$e^{\sfq/z}$ gives
\[
 \mathrm{PF}(z)
 =
 \prod_i(-\Delta+zD+\alpha_i z)(1-\Delta+zD)
 -(1-\Delta)\prod_i(-\Delta+zD+\beta_i z)
 =\sum_{j\geq1}z^j\mathrm{PF}_j.
\]
Using $[D,\Delta]=-(1-\Delta)$ and \eqref{eq:conifold-condition},
\(\mathrm{PF}_1=
 (-1)^{n-1}\Delta^{n-1}(1-\Delta)(\Delta\partial_\Delta+2).
\)
Thus, for $\mathcal R=\sum_{k\geq0}\mathcal R_kz^k$,  \vspace{-0.2cm}
\[
 \mathrm{PF}_1\mathcal R_k
 =
 -\sum_{j=2}^{\min(n+1,k+1)}
 \mathrm{PF}_j\mathcal R_{k+1-j}.
\]
The homogeneous ambiguity at every step is a multiple of $\Delta^{-2}$.
The conditions \eqref{eq:extended-normalization} therefore determine
all $\mathcal R_k$ uniquely.
\end{proof}

\begin{remark}
A closely related Laplace-type transform was considered in \cite{DH}.
\end{remark}

\begin{theorem}\label{thm:rankone}
Let $\omega(\Delta)\in\Delta+\Delta^2\mathbb C[[\Delta]]$, and let $T_\omega$ be the cone-Laplace translation of $\omega$ defined in Proposition--Definition~\ref{defn:conifoldlaplace}. Then the CohFT $T_\omega\Omega^{\mathrm{pt}}$ has the  conifold gap property with respect to $\omega$, with $\kappa=1$, in the sense of Definition~\ref{def:conifold-reg-gap}.

\end{theorem}

\begin{proof}
To apply Theorem \ref{thm:rc-class-gap}, we need to   find  $\gamma^{[1]},\gamma^{[0]}\in\CC[[\Delta,z]]$ such that
\begin{equation}\label{eq:T-diff}
 T_\omega-T_{-\log\sfq}
 =\gamma^{[1]}+\gamma^{[0]}E=O(z^2),
 \qquad \gamma^{[0]}(\Delta,0)=\omega+\log\sfq.
\end{equation}
Define
\[
 u(x):=\omega\!\left(\frac{x}{1+x}\right)-\log(1+x)
 \in x^2\CC[[x]].
\]
Applying $\Delta-zD$ and using \eqref{eq:general-intertwining}
and \eqref{eq:rc-E-Ttau} give asymptotic expansions
\[
 \begin{aligned}
 e^{-\sfq/z}\mathcal H\!\left[\frac{\log x}{2\pi i}\right]
 \sim 1+E(\Delta,-z),\quad
 e^{-\sfq/z}\mathcal H\!\left[
 \frac{\log x\log(1+x)}{2\pi i}\right]
 \sim z+T_{-\log\sfq}(-z).
 \end{aligned}
\]
By \eqref{eq:general-intertwining}, Applying \(-z\partial_\Delta\) to \(e^{-\sfq/z}\mathcal H[\cdot]\) inserts a factor \(x\) into the integrand.
For polynomial $u$,
definition of cone-Laplace transformation, together with $u=O(x^2)$ give
\begin{equation}\label{eq:rc-Tomega}
 T_\omega-T_{-\log\sfq}
 =u(z\partial_\Delta)(1+E)=u(z\partial_\Delta)E=O(z^2).
\end{equation}
It   extends to formal $u$, since the 
$z^k$-coefficients depend only on $u\bmod x^{k+1}$
by \eqref{eq:Rcoeff}.

Next, we show (\ref{eq:rc-Tomega}) can be written in the form of (\ref{eq:T-diff}).
Use \eqref{eq:rc-E-Ttau} to write
\begin{equation}\label{eq:rc-pair-recursion}
 (z\partial_\Delta)^mE=a_mE+b_m,\quad
 a_{m+1}=\frac{\Delta}{1-\Delta}a_m+z\partial_\Delta a_m,\quad
 b_{m+1}=z\partial_\Delta b_m-a_m.
\end{equation}
Induction gives $[z^k]a_m,[z^k]b_m=O(\Delta^{m-2k-1})$.
Thus, for fixed $r,k$, the $\Delta^r z^k$ coefficients of
$a_m$ and $b_m$ vanish for sufficiently large $m$, 
so the coefficients of $1$ and $E$ in $u(z\partial_\Delta)E$  are well-defined.
Collecting the coefficients of $1$ and $E$, we obtain
$u(z\partial_\Delta)E=\gamma^{[1]}+\gamma^{[0]}E$ with
$\gamma^{[1]},\gamma^{[0]}\in\CC[[\Delta,z]]$.
At $z=0$, the recurrence \eqref{eq:rc-pair-recursion} gives
$\gamma^{[0]}(\Delta,0)=u\!\left(\frac{\Delta}{1-\Delta}\right)$. 

Apply Theorem~\ref{thm:rc-class-gap} with this $(\gamma^{[1]},\gamma^{[0]})$ and
$n=0$.
For  $g\ge2$, (\ref{eq:rc-class-gap}) gives
\[
 \int_{\Mbar_g}(T_\omega\Omega^{\pt})_{g,0}
 =\frac{(-1)^gB_{2g}}{2g(2g-2)}\,\omega^{2-2g}+O(\Delta^0).
\]
\end{proof}

\subsection{Completion of the proof of Theorem~\ref{main2}} \label{proofmain2}

Finite generation follows from the reconstruction over $\mathscr R$
in Section~\ref{QDEsystem}. 
Theorem \ref{Omega01} gives (i), and Section~\ref{QDEsystem} verifies (ii). The oscillatory-integral result in \cite[Lemma~6.2 and Proposition~6.1]{NMSP2} at $N=1$, together with the uniqueness in Proposition \ref{defn:conifoldlaplace}, identifies the normalized MSP point series with the series defining the cone-Laplace translation $T_{\omega_1}$, namely $c^{-1}\mathcal R_{\mathrm{pt}}=\mathcal R_{\omega_1}$. {At $\mathfrak t=1$, the dilaton factor is $c^{2-2g}=5^{1-g}$. Theorem \ref{thm:rankone} and the dilaton equation then give (iii). Thus all three conditions of Theorem~\ref{thm:main} hold, and the conifold gap follows.  \qed

\begin{example}\label{ex:translation-family}
The cone-vertex gap theorem applies to the following family, which contains
the cone-Laplace translations of all one-parameter models
considered in this paper.

Put $\Delta=1-\sfq$ and $D=\sfq\partial_{\sfq}$. Let $\omega_p$ be the normalized solution of the following one-parameter deformation of the resolved-conifold Picard–Fuchs equation:
\begin{equation}\label{eq:rc-p-deform-PF}
\big[ D^2- \sfq\cdot (D^2 -p) \big]\omega_p=0,
\qquad \omega_p(\Delta)=\Delta+O(\Delta^2).
\end{equation}
Define $T_k(z)=[p^k]T_{\omega_p}(z)$. For arbitrary parameters
$\mathbf p=(p_1,p_2,\ldots)$, with $p_0=1$, set
\begin{equation}\label{eq:rc-pk-T-omega}
T(z)=\textstyle \sum_{k\ge0}p_kT_k(z),
\qquad
\omega(\Delta;\mathbf p)=\sum_{k\ge0}p_k[p^k]\omega_p(\Delta).
\end{equation}
Then $T=T_\omega$ by linearity, and $T_k=O(z^{k+1})$.
Since $[p^k]\omega_p$ has leading term
$\frac{(-1)^k}{k!(k+1)!}\Delta^{k+1}$, every normalized formal coordinate
determines a unique parameter sequence. 

Inverting the mirror map  $\omega(\Delta;\mathbf p)$  gives
\[
\Delta=\omega+\frac{p_1-1}{2}\omega^2
+\frac{6p_1^2-7p_1-p_2+2}{12}\omega^3+O(\omega^4).
\]
The normalized vanishing periods give the specializations
\renewcommand{\arraystretch}{1.3}
\setlength{\arraycolsep}{3.5pt}
\[ { \small
\begin{array}{c|cccccc|cccc}
 &\multicolumn{6}{c|}{\mathrm{GW}}
 &\multicolumn{4}{c}{\mathrm{FJRW}\ (\sfq=t^k/r)}\\[2pt]
 &\substack{\text{resolved}\\\text{conifold}}
 &\text{quintic}&X_6&X_8&X_{10}&\text{local }\mathbb P^2
 &W_5&W_6&W_8&W_{10}\\ \hline
p_1&0&-\frac25&-\frac{13}{36}&-\frac{11}{32}&-\frac3{10}&-\frac29
 &\frac25&\frac{13}{36}&\frac{11}{32}&\frac3{10}\\[4pt]
p_2&0&\frac{14}{25}&\frac{637}{1296}&\frac{473}{1024}&\frac{39}{100}&\frac{22}{81}
 &-\frac6{25}&-\frac{299}{1296}&-\frac{231}{1024}&-\frac{21}{100}\\[4pt]
p_3&0&-\frac{1138}{625}&-\frac{73873}{46656}&-\frac{48419}{32768}&-\frac{3087}{2500}&-\frac{602}{729}
 &\frac{338}{625}&\frac{25201}{46656}&\frac{17443}{32768}&\frac{1287}{2500}
\end{array}}
\]
with higher parameters determined by the corresponding period.
Setting all $p_k=0$ for $k\ge1$ gives $\omega=-\log(1-\Delta)$.

Write $F_g(\omega;\mathbf p)=\int_{\overline{\mathcal M}_g}
(T\Omega^{\mathrm{pt}})_{g,0}$ after substituting the inverse mirror map.
For the point theory with metric $1$, Theorem~\ref{thm:rankone} gives
\[
F_g(\omega;\mathbf p)
=\frac{(-1)^gB_{2g}}{2g(2g-2)\,\omega^{2g-2}}+O(1),
\qquad g\ge2.
\]
Expanding the point correlators in the conifold flat coordinate $\omega$ gives
\[ \small
\begin{aligned}
F_2(\omega;\mathbf p)={}&-\frac{1}{240\omega^2}
+\frac{64p_1^2+53p_1-15p_2+34}{3840}+O(\omega),\\
F_3(\omega;\mathbf p)={}&-\frac{1}{1008\omega^4}
+\frac{1}{2580480}\Bigl(
120960p_1^4+112640p_1^3-74240p_1^2p_2 +87978p_1^2-39868p_1p_2 \\ &\quad+4608p_1p_3+48616p_1
+4410p_2^2-12353p_2+1458p_3-105p_4+15528
\Bigr)+O(\omega).
\end{aligned}
\]
Thus the parameters enter only the regular part in this coordinate.

\end{example}

$$$$

\section{Applications to one-parameter models}
\label{sec:applications}

\subsection{A uniform genus-zero algorithm for one-parameter Calabi-Yau threefolds}\label{sec:uniform-genus-zero}
Let $X$ have $h^{1,1}=1$ and hypergeometric period equation
$[D^4-\mathsf{q}\prod_{a=1}^4(D+\beta_a)]\omega=0$, where
$D=q\partial_q$, $\mathsf{q}=rq/\ft$, $\Delta=1-\mathsf{q}$, and the
multiset $\{\beta_a\}\subset(0,1)\cap\mathbb Q$ is invariant under
$\beta\mapsto1-\beta$; in particular,
$\sum_a\beta_a=2$.  This includes the models in
\cite[Section~4]{HKQ}.  Put $c=r\prod_a\beta_a$ and
$\kappa=\int_XH^3$. 

We consider  the master  
$I$-function 
\begin{equation}\label{eq:models-I}
I^M=z\sum_{d\ge0}(rq)^d\prod_{m=1}^d
\frac{\prod_{a=1}^4(p+(m-1+\beta_a)z)}
{(p+mz)^4(p+\ft+mz)},\qquad p^4(p+\ft)=0.
\end{equation}
Deleting $p+\ft+mz$,  setting $p=H$, and multiplying by
$(q/\ft)^{H/z}$ gives $I^X=\sum_{j=0}^3I_jH^jz^{1-j}$.

Assume that $X$ admits an MSP theory with the mirror and QRR comparisons of Section~\ref{sec:msp} and polynomial marked $[0,1]$ correlators.
We consider the three conditions of Theorem~\ref{thm:main}.
\begin{itemize} \addtolength{\itemsep}{0.3em}
    \item   The localization together with QRR gives the 
identification 
$$\Omega^M=R T_R(\Omega^{X,\tau}\oplus\Omega^{\mathrm{pt}})$$ with
symplectic $R$. Here $\tau_0=-c\, q$, where $I^M(q,z)=z+c\, q+O(z^{-1})$; the remaining components of $\tau$ are obtained from the formulas in Section~\ref{QDEsystem} by replacing the quintic periods with those of $X$.

\item  
The local-global QDE system follows from the reconstruction in Section \ref{QDEsystem}, with  $120$ and $5^5$ replaced by $c$ and $r$, respectively, and with the periods of $X$ in place of the quintic periods. In particular, $\tau_{\mathrm{pt}}=(r-c)q$.

\item  
At the $\pt$ component,  $p=-\ft$. The series \eqref{eq:models-I} satisfies the extended Picard–Fuchs equation of Proposition~\ref{defn:conifoldlaplace},  with $z$ replaced by $z/\mathfrak t$.
The generalization of the oscillatory-integral argument of \cite[Section~6]{NMSP2}, gives the normalization for the asymptotic solution, so that the positive-order coefficients satisfy the required vanishing at infinity. The uniqueness in Proposition~\ref{defn:conifoldlaplace} identifies this solution with the normalized conifold Laplace solution, and Theorem~\ref{thm:rankone} proves (iii).
\end{itemize}
The reconstruction over the corresponding period ring also gives
finite generation, as in Section~\ref{QDEsystem}. Both finite
generation and the gap argument apply to the other models in
\cite{HKQ} once the required MSP comparisons and polynomiality are established.

For a degree-$k$ Calabi-Yau hypersurface $X_k\subset\mathbb P(a_1,\ldots,a_5)$, where $k=\sum_i a_i$, set
\[
r=\frac{k^k}{\prod_i a_i^{a_i}},
\qquad
c=\frac{k!}{\prod_i a_i!},
\qquad
\kappa=\int_{X_k}H^3=\frac{k}{\prod_i a_i}.
\]
The quintic threefold and the  models studied in \cite{LeiPoly} have the following data:
\[
\begin{array}{c|c|c|c}
k&(a_1,\ldots,a_5)&(\beta_1,\ldots,\beta_4)&\kappa\\ \hline
5&(1,1,1,1,1)&(1,2,3,4)/5&5\\
6&(1,1,1,1,2)&(1,2,4,5)/6&3\\
8&(1,1,1,1,4)&(1,3,5,7)/8&2\\
10&(1,1,1,2,5)&(1,3,7,9)/10&1
\end{array}
\]The MSP moduli spaces are constructed in \cite{CGLLZ}, and their localization and polynomiality properties are studied in \cite[Sections~2-4]{LeiPoly}, following the quintic case. The point-asymptotic argument of \cite[Lemma~5.5 and Section~5.2]{LeiPoly} also applies at $N=1$, since the auxiliary degree-$k$ theory on $\mathbb P(a_1,\ldots,a_5,1)$ remains Fano.

The algorithm above verifies the hypotheses of
Theorem~\ref{thm:main}.
We obtain
\begin{equation}\label{eq:models-gap}
F_g^{X_k,c}(\tau_c)=
\frac{(-1)^{g-1}B_{2g}}
{\kappa^{g-1}\,2g(2g-2)\,\tau_c^{2g-2}}+O(1),
\qquad \tau_c=\omega_1/\omega_0,\quad g\ge2.
\end{equation}

\subsection{Local $\mathbb P^2$}
\label{sec:local-p2-application}
Put $X=\operatorname{Tot}(\mathcal O_{\mathbb P^2}(-3))$ and
$\kappa=-1/3$. Complete the local theory to a rank-four  CohFT
by adjoining a formal degree-three class $H^3$ and setting
\[
\mathscr H_X=\langle1,H,H^2,H^3\rangle,\qquad
\eta_X(H^i,H^j)=\kappa\delta_{i+j,3}.
\]
Let $F_g^+$ be the positive-degree local potentials, write
$t=\log Q$, and set
$\widetilde F_0=F_0^+-t^3/18$, $\widetilde F_1=F_1^+-t/12$, and
$\widetilde F_g=F_g^+$ for $g\ge2$.
For stable $(g,n)$ with no unit insertions, define
\[
\Omega^X_{g,n}(H^{\otimes n})
=\partial_t^n\widetilde F_g(t)\cdot [\pt]_{g,n},\qquad
\Omega^X_{g,n}(\ldots,H^i,\ldots)=0\quad(i=2,3),
\]
where $[\pt]_{g,n}$ is the normalized generator of the top-degree cohomology of \(\overline{\mathcal M}_{g,n}\) such that
\(\int_{\Mbar_{g,n}}[\mathrm{pt}]_{g,n}=1\), and the first
formula includes $n=0$, $g\ge2$.
Extend by the flat-unit axiom, with
$\Omega^X_{0,3}(1,a,b)=\eta_X(a,b)$ and
$\Omega^X_{1,1}(1)=4$.
These tensors define a CohFT of Calabi-Yau threefold type with
potentials $\widetilde F_g(t)$.

The master construction uses
$W=\operatorname{Tot}(\mathcal O_{\mathbb P^3}(-3))$, equivalently
the inverse-Euler twist by $\mathcal O_{\mathbb P^3}(-3)$.
The torus scales the fourth homogeneous coordinate, with
$p|_X=H$ and $p|_{\mathrm{pt}}=-\ft$; its fixed targets are $X$ and
a point. Complete the master state space to
\[
\mathscr H_M=\mathbb Q(\ft)[p]/(p^4(p+\ft)),\qquad
\eta_M(a,b)=\kappa[p^4](ab),
\]
where the coefficient is taken after reduction to degree at most
four. The added class $p^3(p+\ft)/\ft$ restricts to $H^3$, and its
quotient recovers the geometric relation $p^3(p+\ft)=0$.

The master $I$-function is \eqref{eq:models-I} with
\[
r=-27,\qquad c=0,\qquad
(\beta_1,\beta_2,\beta_3,\beta_4)=(0,\tfrac13,\tfrac23,1),
\qquad \mathsf q=-27q/\ft.
\]
Construct the shift $\tau_X$ and the symplectic $R$-matrix on the
completed state spaces as in Section~\ref{sec:uniform-genus-zero},
using these parameters and the QRR factors of $W$. Define
\[
\Omega^M=R T_R(\Omega^{X,\tau_X}\oplus\Omega^{\mathrm{pt}}),
\]
where $\Omega^{\mathrm{pt}}$ has metric $\kappa$ and the divisor
shift is absorbed into the Novikov variable $Q$.

Localization identifies the positive-degree marked ancestor
correlators of $\Omega^M$ with those of $W$ under the natural
projection of state spaces.
Correlators containing $p^3(p+\ft)/\ft$, which restricts to the
formal $H^3$, have zero positive-degree part and are therefore
$q$-independent.
Together with the geometric degree bound, this proves polynomiality
of all stable marked ancestor correlators of $\Omega^M$, hence (i).
The verifications of (ii) and (iii) are the same as in
Section~\ref{sec:uniform-genus-zero}.
The reconstruction over the corresponding period ring gives finite
generation for $F_g^+$, $g\ge2$, and Theorem~\ref{thm:main} gives
the conifold gap for local $\mathbb P^2$.

\subsection{FJRW theory of Fermat singularities}\label{sec:fjrw-application}
Use the same notations as in Section~\ref{sec:uniform-genus-zero}.
Let $W_k=\sum_{i=1}^5x_i^{k/a_i}$ be the Fermat polynomial defining the corresponding Calabi-Yau hypersurface in Section~\ref{sec:uniform-genus-zero}. We consider the FJRW theory of $(W_k,G)$, where $G=\langle J_W\rangle$ and $J_W=\operatorname{diag}(e^{2\pi i a_1/k},\ldots,e^{2\pi i a_5/k})$. Denote its CohFT by $\Omega^{\mathrm{LG}}$.

 Write $\varphi_j$ for the narrow state in the
$J_W^{k\beta_{j+1}}$ sector, $0\le j\le3$, with
$\deg\varphi_j=j$ and
$\eta^{\mathrm{LG}}(\varphi_i,\varphi_j)=\delta_{i+j,3}/k$.
Its dual basis is $\varphi^i=k\varphi_{3-i}$, extended by zero on the point.
Let $I_j$ denote the FJRW $I$-functions
in this basis \cite[Chapters~1 and 3]{Z}, normalized by
$I_j=c_jt^{k\beta_{j+1}}(1+O(t^k))$, where $t$ is the complex parameter near the orbifold point, $c_0=1$ and $c_j$ are rational numbers determined by $k,a_i,\beta_l$.  We use
\[
\mathsf q=\frac{t^k}{r} ,\qquad \Delta=1-\mathsf q,
\qquad D=\frac tk\partial_t ,\qquad \tau=I_1/I_0.
\]

We verify the same three conditions as in
Section~\ref{sec:uniform-genus-zero}.
\begin{itemize}
\item[(i)]  Taking $d_0=0$ in  MSP theory, we get
MSP $[1,\infty]$ theory, with a weight one torus action scaling $\nu$ field. Then localization and QRR give
\cite[Chapters~2-4]{Z}
\begin{equation}\label{eq:fjrw-reconstruction}
\Omega^M=RT_R(\Omega^{W,\tau_W}\oplus\Omega^{\mathrm{pt}}),
\qquad \Omega^W_{g,n}=t^{2-2g-n}\Omega^{\mathrm{LG}}_{g,n}.
\end{equation}
Here $\tau_W=0\cdot\varphi_0+t(I_1/I_0)\varphi_1+\cdots$ is the Dijkgraaf-Witten map, and the rescaled CohFT $\Omega^W$ has unit $t\varphi_0$.
The virtual dimension bound implies that the even  ancestor correlators are Laurent
polynomials in $t$. Hence the conifold regularity holds. 
\item[(ii)] The MSP $I$-function restricted to $d_0=0$
 is given by (see \cite[Chapter~3]{Z})
\begin{align}\label{eq:fjrw-master-I}
I^M={}&z\sum_{d\ge0}\mathsf q^d\prod_{m=1}^d
\frac{(-1+(m-1)z)^4}
{mz\prod_{a=1}^4(-1+(m-\beta_a)z)}\,\mathbf1_{\mathrm{pt}}\\
&+z\sum_{j=0}^3\sum_{d\in\beta_{j+1}+\mathbb Z_{\ge0}}t^{kd}
\frac{ \prod_{i=1}^5
\prod_{\substack{0<b<a_i d\\\langle b\rangle=\langle a_i d\rangle}}(bz)}
{  z^{kd-1}(kd-1)!
\prod_{\substack{0<b\le d\\\langle b\rangle=\langle d\rangle}}(1+bz)}
\,\varphi_j. \nonumber
\end{align}
Here $\langle\cdot\rangle$ denotes fractional part. Direct computation of
Birkhoff factorization of $I^M$ gives
$A^M\in\operatorname{Mat}_5(\mathbb Q[t,t^{-1}])$ and
\[
\det[\mathbf1_M,A^M\mathbf1_M,\ldots,(A^M)^4\mathbf1_M]
=c_k t^{2k},\qquad c_k\ne0.
\]
Thus $A^M,A^{\mathrm{pt}}=-1+\mathsf q$ and the inverse of this matrix
are conifold regular.
The QRR formulas for $\mathcal R_0,\mathcal R_1$ are differential
polynomials in $I_0$ and $D\tau$, as in GW.
The substitutions $I_i\mapsto\omega_i$ show that both $\mathcal R_0$ and $\mathcal R_1$ are conifold regular, with
$\mathcal R_0|_{z=0}=\omega_0$ and
$D(\omega_1/\omega_0)=-1+O(\Delta)$ a unit.
Hence Proposition~\ref{prop:decompose-R} applies.
\item[(iii)] The Picard-Fuchs equation for \eqref{eq:fjrw-master-I}
is \eqref{eq:general-extended} with left shifts $-\beta_a$ and
zero right shifts, and $4-\sum_a\beta_a=2$.
With the QRR normalization \eqref{eq:extended-normalization},
uniqueness in Proposition~\ref{defn:conifoldlaplace} gives
$\mathcal R_{\mathrm{pt}}=\cR_{\omega_1}$;
Theorem~\ref{thm:rankone} then proves (iii).
\end{itemize}

Finally, the translated primary  potential of $\Omega^W$ is 
$$(I_0/t)^{2-2g}t^{2-2g}F_g^{\mathrm{LG}}(\tau)
=I_0^{2-2g}F_g^{\mathrm{LG}}(\tau).$$
The reconstruction gives $F_g^{\mathrm{LG}}(\tau)\in\mathscr R$
for $g\ge2$, with $\mathscr R$ the corresponding FJRW period ring.
The proof of Theorem~\ref{thm:main} then gives the conifold expansion:
\begin{equation}\label{eq:fjrw-gap}
F_g^{\mathrm{LG},c}(\tau_c)
=\frac{\kappa^{1-g}(-1)^{g-1}B_{2g}}{2g(2g-2)\tau_c^{2g-2}}+O(1),
\qquad \tau_c=\omega_1/\omega_0,\quad g\ge2.
\end{equation}

\subsection{Summary}
The diagram below summarizes how MSP theory transfers the cone-vertex gap to the GW and FJRW theories. For quintic GW theory, we set $d_\infty=0$ and restrict
the localization sum to the $[0,1]$ part, excluding level-$\infty$
vertices \cite{NMSP2}. For FJRW theory, we set $d_0=0$ and naturally obtain the $[1,\infty]$-theory \cite{GR,Z}. These two parts connect
the common level-one point theory to the level-zero GW theory
and the level-$\infty$ FJRW theory, respectively.
In each case, the corresponding translated point theory (cone-vertex theory) has
the conifold gap property. Conifold regularity of the restricted
MSP theory, together with the vanishing of all non-leading graphs after removing the regular $R$-matrix factor, then transfers the gap
to the GW or FJRW theory.
\vspace{0.5cm}

\begin{tikzpicture}[x=1cm,y=1cm]

  \begin{scope}[shift={(8.30,5.55)}]
    \MSPGeometry
  \end{scope}
\end{tikzpicture}

$$$$

\appendix

\addtocontents{toc}{\protect\setcounter{tocdepth}{-1}}

\section{The computation of the coordinate}
\label{app:dwmap}

For completeness, 
we prove the formula for the Dijkgraaf-Witten coordinate $\tau_Q$, which was presented  in p.~623 of \cite{NMSP2} without a proof.   
We set $\ft=1$ throughout this appendix.

\begin{lemma} \label{formulatauq}
The coordinate $\tau_Q$ defined in  \eqref{birk} is given by
$$
\tau_Q= - 5!q+\tau_1 H  +\tau_2 H^2 +\tau_3 H^3
$$
with $\tau_1= I_1/I_0-\log q$, $\tau_2 = - \frac{1}{2} \, I_{11} I_{22}$ and $\tau_3 =  I_{11}^2I_{22} \big( -\frac{1}{12}+\frac{Y}{4}-\frac{B_1}{2}-\frac{A_1}{4} \big)$
\end{lemma}

\begin{proof} We complete the calculation in Section~\ref{quintictauq} by cancelling the negative
powers of $z$ in the quintic unit column. 
The scalar factor $e^{120q/z}$ in \eqref{msp-unit} at $-z$ is
cancelled by $\tau_0=-120q$. The divisor equation gives
$\tau_1=I_1/I_0-\log q$. It remains to determine $\tau_2,\tau_3$.

Only the first two operators from Section~\ref{quintictauq} are needed:
\[
\mathfrak D_1(D)=\frac{D^2}{2}+\frac D2+\frac1{12},
\qquad
\mathfrak D_2(D)=\frac{D^4}{8}+\frac{5D^3}{12}
+\frac{5D^2}{12}+\frac D8+\frac1{288}.
\]
The recurrence for $I_j^{[k]}$ gives
\[
\begin{aligned}
I_2^{[\mathfrak D_1]}&=\frac12 I_0I_{11}I_{22},\\
I_1^{[\mathfrak D_1]}&=-I_0I_{11}
\left(B_1+\frac{A_1}{2}+\frac12\right),\\
I_3^{[\mathfrak D_2]}&=-I_0I_{11}^2I_{22}
\left(\frac Y4+\frac16\right).
\end{aligned}
\]
For the last equality, the same recurrence and
$I_{22}=Y/(I_0^2I_{11}^2)$ give
$I_3^{[3]}=-I_0I_{11}^2I_{22}$ and
$I_3^{[4]}=-2I_0I_{11}^2I_{22}(Y-1)$.

Since $I_j^{[k]}=0$ for $k<j$ and
$\deg_D\mathfrak D_m\le2m$, the only possible negative powers in
the sum from Section~\ref{quintictauq} are
$H^2z^{-1}I_2^{[\mathfrak D_1]}$ and
$H^3z^{-1}I_3^{[\mathfrak D_2]}$.
After multiplication by $1+(\tau_2H^2+\tau_3H^3)/z$, the only
possible poles are still the monomials $H^2z^{-1},H^3z^{-1}$,
since $H^4=0$. In the full unit column, dividing their coefficients
by $I_0$ gives the two equations
\begin{equation}\label{msp-poles}
\begin{aligned}
0&=\tau_2+\frac{I_{11}I_{22}}2,\\
0&=\tau_3-I_{11}^2I_{22}
\left(\frac Y4+\frac16\right)
-\tau_2I_{11}
\left(B_1+\frac{A_1}{2}+\frac12\right).
\end{aligned}
\end{equation}
Solving these equations and using $I_{11}^2I_{22}=Y/I_0^2$
gives the stated formulas.
\end{proof}

The same recurrence proves the ring property needed in Section~\ref{quintictauq}.
\begin{lemma}\label{lem:I-ring}
For $0\le j\le3$ and $k\ge0$,
\begin{equation}\label{msp-ring}
\frac{I_j^{[k]}}{I_0I_{11}\cdots I_{jj}}
\in\mathbb Q[A_1,B_1,B_2,B_3,Y],
\end{equation}
where the denominator is $I_0$ when $j=0$.
The same assertion holds with $I_j^{[k]}$ replaced by
$I_j^{[\mathfrak D_m]}$.
In particular,  
$I_j^{[k]},I_j^{[\mathfrak D_m]}\in\mathscr R$, and every coefficient
of $\left.R(z)^{-1}\mathbf1_M\right|_Q$ belongs to $\mathscr R$.
\end{lemma}

\begin{proof}
The five-generator ring in \eqref{msp-ring}
is closed under $D$ \cite[Section~0.1]{NMSP2}.
The four denominators have logarithmic derivatives
\[
B_1,\qquad B_1+A_1,\qquad
Y-1-B_1-A_1,\qquad Y-1-B_1.
\]
Divide the recurrence of Section~\ref{quintictauq} by the corresponding
denominator. The factor $I_{jj}$ in the last term cancels the
ratio of consecutive denominators, so induction on $k$, starting
from $I_j^{[0]}=\delta_{j0}I_0$, proves
\eqref{msp-ring}.
Rational linearity gives the assertion for $\mathfrak D_m$.

The four denominators are
$I_0$, $I_0I_{11}$, $Y/(I_0I_{11})$, and $Y/I_0$,
all belong to $\mathscr R$, as do $\tau_2,\tau_3$.
Each coefficient of the unit-column formula in Section~\ref{quintictauq}
contains only finitely many terms, which proves the last assertion.
\end{proof}

\end{document}